\documentclass[a4paper, 11pt]{article} 
\usepackage[headings]{fullpage}   
\usepackage{amsfonts}
\usepackage{amsmath} 
\usepackage{amssymb}
\usepackage{graphicx}
\usepackage{amsthm}
\usepackage[dvipsnames,svgnames]{xcolor}
\usepackage[colorlinks=true,linkcolor=RoyalBlue,urlcolor=RoyalBlue,citecolor=PineGreen]{hyperref}
\usepackage[nameinlink]{cleveref}
\usepackage{tikz}
\usepackage{optidef}
\usetikzlibrary{angles}
\usetikzlibrary{calc}
\usetikzlibrary {intersections,through}
\usepackage{enumerate}
\usepackage[title]{appendix}

\hypersetup{linktocpage=true}
\newtheorem{theorem}{Theorem}[section]
\newtheorem{lemma}[theorem]{Lemma}

\newtheorem{proposition}[theorem]{Proposition}
\newtheorem{corollary}[theorem]{Corollary}
\theoremstyle{definition}
\theoremstyle{definition}
\theoremstyle{definition}
\theoremstyle{definition}\newtheorem{claim}{Claim}

\newcommand{\rank}{\operatorname{rank}}

\title{5-regular graphs and the 3-dimensional rigidity matroid}
\author{Rebecca Monks and Anthony Nixon}
\date{\today}

\begin{document}

\maketitle

\begin{abstract}
A bar-joint framework $(G,p)$ in Euclidean $d$-space is rigid if the only edge-length-preserving continuous motions arise from isometries of $\mathbb{R}^d$. In the generic case, rigidity is determined by the generic $d$-dimensional rigidity matroid of $G$. The combinatorial nature of this matroid is well understood when $d=1,2$ but open when $d\geq 3$.
Jackson and Jord\'an \cite[Theorem 3.5]{JJbounded} characterised independence in this matroid for connected graphs with minimum degree at most $d+1$ and maximum degree at most $d+2$. 
Their characterisation is known to be false for $(d+2)$-regular graphs when $d\geq 4$ but when $d=3$ it remained open. Indeed they conjectured \cite[Conjecture 3.10]{JJbounded} that their characterisation extends to 5-regular graphs when $d=3$. The purpose of this article is to prove their conjecture. That is, we prove that every 5-regular graph that has at most $3n-6$ edges in any subgraph on $n\geq 3$ vertices is independent in the generic 3-dimensional rigidity matroid. 
\end{abstract}

\section{Introduction}

A bar-joint \emph{framework} $(G,p)$ is the combination of a finite, simple graph $G=(V,E)$ and a map $p:V\rightarrow \mathbb{R}^d$. The framework is \emph{$d$-rigid} if every edge-length-preserving continuous deformation of the vertices arises from an isometry of $\mathbb{R}^d$. In general it is a computationally challenging problem to determine if a given framework is $d$-rigid \cite{abbott}. This motivates restricting attention to generic frameworks: $(G,p)$ is \emph{generic} if the set of coordinates of $p$ forms an algebraically independent set over $\mathbb{Q}$. 

As is typical in the literature we will analyse generic framework rigidity by a linearisation known as infinitesimal rigidity. Concretely, we define the \emph{rigidity matrix} $R_d(G,p)$ to be the $|E|\times d|V|$ matrix with rows indexed by the edges and $d$-tuples of columns indexed by the vertices. The row for an edge $uv\in E$ is zero everywhere except the $d$-tuple of columns corresponding to $u$ where the entries are $p(u)-p(v)$ and the $d$-tuple of columns corresponding to $v$ where the entries are $p(v)-p(u)$. It is easy to see that, when $|V|\geq d$, $\rank R_d(G,p)\leq d|V|-{d+1\choose 2}$. 

It was shown by Asimow and Roth \cite{asi-rot} that a generic framework $(G,p)$ is $d$-rigid if and only if $|V|<d$ and $G$ is complete or $|V|\geq d$ and $\rank R_d(G,p)= d|V|-{d+1\choose 2}$. This equivalence implies that $d$-rigidity is a generic property. That is, if $(G,p)$ is $d$-rigid then every generic framework $(G,q)$ in $\mathbb{R}^d$ is $d$-rigid. 

The \emph{generic $d$-dimensional rigidity matroid} $\mathcal{R}_d(G,p)$ is the row matroid of $R_d(G,p)$. This depends only on $G$ and we will drop the $G$ and refer to $\mathcal{R}_d$ when the context is clear. We will say that $G$ is \emph{$\mathcal{R}_d$-independent} if the edge set of $G$ is independent in $\mathcal{R}_d$, i.e. if $R_d(G,p)$ has linearly independent rows, and that $G$ is \emph{$\mathcal{R}_d$-dependent} if it is not $\mathcal{R}_d$-independent. Further $G$ is \emph{minimally $\mathcal{R}_d$-rigid} if some generic framework $(G,p)$ is $d$-rigid and $G$ is $\mathcal{R}_d$-independent. We will also use $r_d$ to denote the rank function of $\mathcal{R}_d$ and \emph{$\mathcal{R}_d$-circuit} to refer to a graph $G$ that is $\mathcal{R}_d$-dependent but $G-e$ is $\mathcal{R}_d$-independent for all edges $e$ of $G$.

When $d=1$, a simple folklore result shows that $\mathcal{R}_1$ is  isomorphic to the cycle matroid and hence $\mathcal{R}_1$-independent graphs are forests and $\mathcal{R}_1$-circuits are cycles. When $d=2$, $\mathcal{R}_2$ can also be described in purely combinatorial terms due to a celebrated result of Pollaczek-Geiringer \cite{PG} known as Laman's theorem \cite{laman}. However when $d>2$ no known combinatorial description is available and obtaining one is a central problem in rigidity theory. See \cite{CJT,GGJN,JJbounded,JJdress,Jor,Tay,Wlong}, inter alia, for partial results in this direction.

A necessary condition for $\mathcal{R}_d$-independence goes back to Maxwell \cite{Max}. For $G=(V,E)$ and $X\subset V$, we let $i_G(X)$, simply $i(X)$ when the graph is clear from the context, denote the number of edges in the subgraph of $G$ induced by $X$. Let us say that $G$ is \emph{$d$-sparse} if $i(X)\leq d|X|-{d+1 \choose 2}$ holds for all $X\subset V$ with $|X|\geq d$. If $G$ is $d$-sparse and $|E|=d|V|-{d+1\choose 2}$ then we say that $G$ is \emph{$d$-tight}. 

\begin{lemma}[Maxwell \cite{Max}]\label{lem:max}
Let $G$ be $\mathcal{R}_d$-independent. Then $G$ is $d$-sparse.    
\end{lemma}

The aforementioned characterisations of $\mathcal{R}_d$-independence when $d\leq 2$ imply that Maxwell's condition is also sufficient when $d=1,2$. However it is no longer sufficient when $d\geq 3$. It is therefore interesting to determine special families of graphs where Maxwell's condition is sufficient to guarantee $\mathcal{R}_d$-independence.
In particular, Jackson and Jord\'an \cite{JJbounded} proved the following result.

\begin{theorem}[{\cite[Theorem 3.5]{JJbounded}}] \label{thm:bounded}
Let $G$ be a connected graph with minimum degree at most $d+1$ and maximum degree at most $d+2$. Then $G$ is $\mathcal{R}_d$-independent if and only if $G$ is $d$-sparse.    
\end{theorem}

When $d>3$ this result is best possible in the sense that the complete bipartite graph $K_{d+2,d+2}$ is $\mathcal{R}_d$-dependent \cite[Theorem 5.2.1]{GSS} but is also $d$-sparse. This leaves open the case when $d=3$ and $G$ is 5-regular. (Here $K_{5,5}$ is an $\mathcal{R}_3$-circuit but is not 3-sparse.) Jackson and Jord\'an conjectured \cite[Conjecture 3.10]{JJbounded} that every 5-regular, 3-sparse graph is $\mathcal{R}_d$-independent. We prove their conjecture in this paper.

\begin{theorem}\label{thm:main}
Let $G=(V,E)$ be a 5-regular graph. Then $G$ is $\mathcal{R}_3$-independent if and only if $G$ is $3$-sparse.    
\end{theorem}

We remark that 5-regular, 3-sparse graphs were already known to be independent in the cofactor matroid \cite{Wat}. Here the cofactor matroid is a closely related matroid to $\mathcal{R}_3$ that arises from the study of bivariate splines, see \cite{CJT} for a definition and details on the relationship between the two matroids. The proof in \cite{Wat} uses the so-called X-replacement operation. It is a central problem in rigidity theory to determine if this operation preserves $\mathcal{R}_3$-independence \cite{GSS,Wlong}. 
We will take an alternative approach using the following geometric operations.

A \emph{($d$-dimensional) vertex split} of a graph $G=(V,E)$ is the operation defined as follows: choose $v\in V$, $x_1,x_2,\dots,x_{d-1}\in N_G(v)$ and 
a partition $N_1,N_2$ of 
pairwise disjoint sets $N_1,N_2$ with $N_1\cup N_2=N_G(v)\setminus \{x_1,x_2,\dots,x_{d-1}\}$; then delete $v$ from
$G$ and add two new vertices $v_1,v_2$ joined to $N_1,N_2$, 
respectively; finally add new edges $v_1v_2,v_1x_1, v_2x_1, v_1x_2,v_2x_2,\dots, v_1x_{d-1},v_2x_{d-1}$. This operation is illustrated in Figure \ref{fig:vsplit}. The reverse operation is known as \emph{edge contraction}.

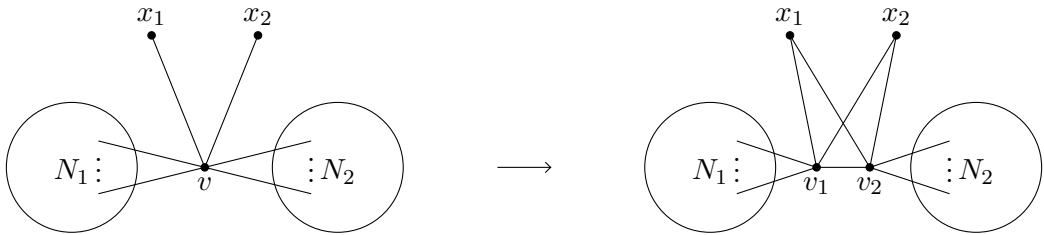
\begin{figure}[ht]
\begin{center}
\begin{tikzpicture}[scale=0.7]
\draw (-3,-5) circle (35pt);
\draw (4,-5) circle (35pt);
\draw (9,-5) circle (35pt);
\draw (-8,-5) circle (35pt);

\draw (-3,-5.5) circle (0pt) node[anchor=south]{$N_2$};
\draw (4,-5.5) circle (0pt) node[anchor=south]{$N_1$};
\draw (9,-5.5) circle (0pt) node[anchor=south]{$N_2$};
\draw (-8,-5.5) circle (0pt) node[anchor=south]{$N_1$};

\filldraw (-5.5,-5) circle (2pt) node[anchor=north]{$v$};
\filldraw (6,-5) circle (2pt) node[anchor=north]{$v_1$};
\filldraw (7,-5) circle (2pt) node[anchor=north]{$v_2$};

\filldraw (-6.5,-2.5) circle (2pt) node[anchor=south]{$x_1$};
\filldraw (-4.5,-2.5) circle (2pt) node[anchor=south]{$x_2$};

\filldraw (5.5,-2.5) circle (2pt) node[anchor=south]{$x_1$};
\filldraw (7.5,-2.5) circle (2pt) node[anchor=south]{$x_2$};

\node at (4.5,-4.9){$\vdots$};
\node at (8.5,-4.9){$\vdots$};

\node at (-3.5,-4.9){$\vdots$};
\node at (-7.5,-4.9){$\vdots$};

\draw[black]
(-5.5,-5) -- (-6.5,-2.5);

\draw[black]
(-5.5,-5) -- (-4.5,-2.5);

\draw[black]
(-5.5,-5) -- (-3.5,-4.5);

\draw[black]
(-5.5,-5) -- (-3.5,-5.5);

\draw[black]
(-5.5,-5) -- (-7.5,-5.5);

\draw[black]
(-5.5,-5) -- (-7.5,-4.5);

\draw[black]
(6,-5) -- (7,-5);

\draw[black]
(6,-5) -- (5.5,-2.5);

\draw[black]
(6,-5) -- (7.5,-2.5);

\draw[black]
(5.5,-2.5) -- (7,-5);

\draw[black]
(7.5,-2.5) -- (7,-5);

\draw[black]
(4.5,-4.5) -- (6,-5);

\draw[black]
(4.5,-5.5) -- (6,-5);

\draw[black]
(7,-5) -- (8.5,-4.5);

\draw[black]
(7,-5) -- (8.5,-5.5);

\draw[black]
(0,-5) -- (1,-5) -- (0.9,-5.1);

\draw[black]
(1,-5) -- (0.9,-4.9);

\end{tikzpicture}
\caption{A schematic of the 3-dimensional vertex split.} \label{fig:vsplit}
\end{center}
\end{figure}

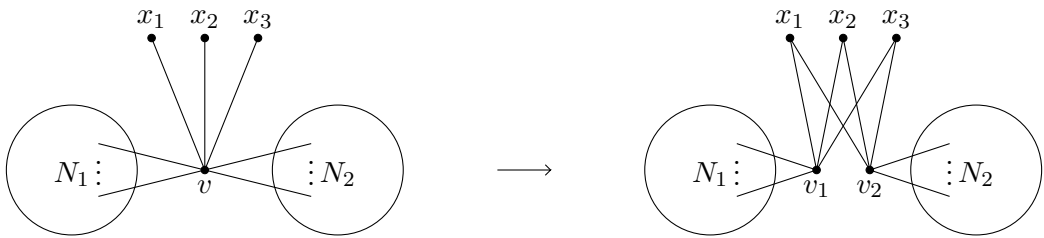
\begin{figure}[ht]
\begin{center}
\begin{tikzpicture}[scale=0.7]
\draw (-3,-5) circle (35pt);
\draw (4,-5) circle (35pt);
\draw (9,-5) circle (35pt);
\draw (-8,-5) circle (35pt);

\draw (-3,-5.5) circle (0pt) node[anchor=south]{$N_2$};
\draw (4,-5.5) circle (0pt) node[anchor=south]{$N_1$};
\draw (9,-5.5) circle (0pt) node[anchor=south]{$N_2$};
\draw (-8,-5.5) circle (0pt) node[anchor=south]{$N_1$};

\filldraw (-5.5,-5) circle (2pt) node[anchor=north]{$v$};
\filldraw (6,-5) circle (2pt) node[anchor=north]{$v_1$};
\filldraw (7,-5) circle (2pt) node[anchor=north]{$v_2$};

\filldraw (-6.5,-2.5) circle (2pt) node[anchor=south]{$x_1$};
\filldraw (-4.5,-2.5) circle (2pt) node[anchor=south]{$x_3$};
\filldraw (-5.5,-2.5) circle (2pt) node[anchor=south]{$x_2$};

\filldraw (5.5,-2.5) circle (2pt) node[anchor=south]{$x_1$};
\filldraw (7.5,-2.5) circle (2pt) node[anchor=south]{$x_3$};
\filldraw (6.5,-2.5) circle (2pt) node[anchor=south]{$x_2$};

\node at (4.5,-4.9){$\vdots$};
\node at (8.5,-4.9){$\vdots$};

\node at (-3.5,-4.9){$\vdots$};
\node at (-7.5,-4.9){$\vdots$};

\draw[black]
(-5.5,-5) -- (-6.5,-2.5);

\draw[black]
(-5.5,-5) -- (-4.5,-2.5);

\draw[black]
(-5.5,-5) -- (-5.5,-2.5);

\draw[black]
(-5.5,-5) -- (-3.5,-4.5);

\draw[black]
(-5.5,-5) -- (-3.5,-5.5);

\draw[black]
(-5.5,-5) -- (-7.5,-5.5);

\draw[black]
(-5.5,-5) -- (-7.5,-4.5);

\draw[black]
(6,-5) -- (5.5,-2.5);

\draw[black]
(6,-5) -- (7.5,-2.5);

\draw[black]
(5.5,-2.5) -- (7,-5);

\draw[black]
(7.5,-2.5) -- (7,-5);

\draw[black]
(4.5,-4.5) -- (6,-5);

\draw[black]
(6.5,-2.5) -- (7,-5);

\draw[black]
(6.5,-2.5) -- (6,-5);

\draw[black]
(4.5,-5.5) -- (6,-5);

\draw[black]
(7,-5) -- (8.5,-4.5);

\draw[black]
(7,-5) -- (8.5,-5.5);

\draw[black]
(0,-5) -- (1,-5) -- (0.9,-5.1);

\draw[black]
(1,-5) -- (0.9,-4.9);

\end{tikzpicture}
\caption{A schematic of the 3-dimensional spider split.} \label{fig:ssplit}
\end{center}
\end{figure}

A \emph{($d$-dimensional) spider split} of a graph $G=(V,E)$ is the operation defined as follows: choose $v\in V$, $x_1,x_2,\dots,x_d\in N_G(v)$ and 
a partition $N_1,N_2$ of 
pairwise disjoint sets $N_1,N_2$ with $N_1\cup N_2=N_G(v)\setminus \{x_1,x_2,\dots,x_{d}\}$; then delete $v$ from
$G$ and add two new vertices $v_1,v_2$ joined to $N_1,N_2$, 
respectively; finally add new edges $v_1x_1, v_2x_1,v_1x_2,v_2x_2,\dots, v_1x_{d},v_2x_{d}$. This operation is illustrated in Figure \ref{fig:ssplit}. The reverse operation is known as \emph{spider contraction}.

\begin{lemma}[{\cite[Proposition 10]{Whi}}]\label{lem:vsplit} 
Let $G$ be $\mathcal{R}_d$-independent and let $G'$ be obtained from $G$ by a vertex split or a spider split. Then $G'$ is $\mathcal{R}_d$-independent.
\end{lemma}

We will also need the so-called isostatic substitution principle.

\begin{lemma}[Whiteley {\cite[Corollary 2.8]{WhiSub}}, see also Finbow, Ross and Whiteley {\cite[Lemma 2.2]{FRW}}]\label{lem:isosub}
 Let $G=(V,E)$ be a graph and let $H$ be a proper induced subgraph of $G$ that is minimally $\mathcal{R}_d$-rigid. Let $H'$ be a minimally $\mathcal{R}_d$-rigid graph on the same vertex set as $H$ and let $G'$ be obtained from $G$ by deleting the edge set of $H$ and adding the edge set of $H'$. Then, for any generic $p$, $\mbox{rank } R(G,p)=\mbox{rank } R(G',p)$, i.e. $G$ is $\mathcal{R}_d$-independent if and only if $G'$ is $\mathcal{R}_d$-independent. 
\end{lemma}

The remainder of the paper is purely combinatorial containing an analysis of when edge contractions and spider contractions preserve $3$-sparsity. We conclude the introduction with a brief outline of the proof of Theorem \ref{thm:main}.

We mention only the difficult sufficiency direction. We assume the theorem is false and consider a minimal counterexample $G$. If $G$ can be obtained from an $\mathcal{R}_3$-independent graph $H$ by vertex splitting or spider splitting then $G$ is not a counterexample by Lemma \ref{lem:vsplit}. In order to know that $H$ is $\mathcal{R}_3$-independent we would like to apply Theorem \ref{thm:bounded}. To this end we will only consider `reducible' reductions; that is, reductions that result in $H$ having a vertex of degree 4. Necessarily $H$ will have one vertex of degree greater than 5 so we cannot immediately apply Theorem \ref{thm:bounded}, thus we first prove an extension of that result in Section \ref{sec:bounded} and this suffices to establish $\mathcal{R}_3$-independence if we can show $H$ is 3-sparse. We establish when reducible reductions are possible in Section \ref{sec:contract}. With that in hand we consider when $H$ is 3-sparse.
When it is not the reduction is said to be non-admissible. The `blockers' in $G$ that result in non-admissibility can take several forms as described in Subsection \ref{subsec:structure}. We prove our main result in Section \ref{sec:proof}. First we apply the isostatic substitution principle, Lemma \ref{lem:isosub}, to reduce these possibilities. We then apply the results of Subsection \ref{subsec:structure} to restrict the structural possibilities for $G$. From there we consider a maximal blocker $Y$ and reducible edges or vertex-pairs with one vertex in $Y$ and one not in $Y$. This gives a second blocker $Z$, and from here, we apply a case analysis based on the interaction between $Y$ and $Z$.

\section{Bounded $d$-sparse graphs}
\label{sec:bounded}

We prove a short extension of Theorem \ref{thm:bounded} to allow one vertex of arbitrary degree. A similar extension was obtained in \cite[Theorem 4.5]{JJdress}. In their case the extended family of graphs allowed two vertices of degree greater than 5 but both needed to be adjacent vertices of degree 6.

\begin{theorem}\label{thm:1highdeg}
Let $G=(V,E)$ be a 2-connected graph with a designated vertex $v$. Suppose $G$ has minimum degree at most $d+1$ and is such that every vertex in $V-v$ has degree at most $d+2$. Then $G$ is $\mathcal{R}_d$-independent if and only if $G$ is $d$-sparse.    
\end{theorem}

Notice that the theorem is best possible in the following sense. There exists graphs with minimum degree $d+1$ and two vertices of degree $d+3$ that are $d$-sparse but not $\mathcal{R}_d$-independent. The smallest such example \cite{GGJN,Jor} is the 2-sum of two copies of $K_5$ (i.e. two copies of $K_5$ glued together along an edge with the common edge deleted from both). This graph is 3-sparse but $\mathcal{R}_3$-dependent.

The technical tools in the proof of Theorem \ref{thm:1highdeg} are the following well used extension operations.

 A graph $G'$ is said to be obtained from another graph $G$ by: a \emph{($d$-dimensional) 0-extension}
if $G=G'-v$ for a vertex  $v\in V(G')$ with $d_{G'}(v)=d$; or a \emph{($d$-dimensional) 1-extension} if $G=G'-v+xy$ for a vertex $v\in V(G')$ with $d_{G'}(v)=d+1$ and $x,y\in N_{G'}(v)$ with $x,y$ non-adjacent.
The inverse operations of 0-extension and 1-extension are called \emph{0-reduction} and \emph{1-reduction}, respectively.

\begin{lemma}[{\cite[Lemma 11.1.1, Theorem 11.1.7]{Wlong}}]\label{lem:01ext}
Let $G$ be $\mathcal{R}_d$-independent and let $G'$ be obtained from $G$ by a 0-extension or a 1-extension. Then $G'$ is $\mathcal{R}_d$-independent.
\end{lemma}

\begin{lemma}[{\cite[Lemma 3.3]{JJbounded}}]
\label{lem:reduclem}
Let $G=(V,E)$ be a $d$-sparse graph. Let $v$ be a vertex of degree $d+1$ in $G$, and $V' = \{x\in N(v):d_G(x)\geq d+3\}$. Suppose that $G[V']$ is a (possibly empty) complete graph. Then there is a 1-reduction at $v$ resulting in a smaller $d$-sparse graph.
\end{lemma}

We will need one more known result.

\begin{lemma}[{\cite[Lemma 11.1.9]{Wlong}}]\label{lem:intbridge}
Let $G_1$, $G_2$ be subgraphs of a graph $G$ and suppose that $G=G_1\cup G_2$.
\begin{enumerate}
\item\label{it:intbridge:indep} 
	If $G_1\cap  G_2$ is $\mathcal{R}_d$-rigid and $G_1,G_2$ are $\mathcal{R}_d$-independent then $G$ is $\mathcal{R}_d$-independent.
\item\label{it:intbridge:rank} 
	If $|V(G_1)\cap V(G_2)| \leq d-1$, $u\in V(G_1)\setminus V(G_2)$ and $v\in V(G_2)\setminus V(G_1)$ then
	$r_d(G+uv)=r_d(G)+1$.
\end{enumerate}
\end{lemma}

\begin{proof}[Proof of Theorem \ref{thm:1highdeg}]
Lemma \ref{lem:max} gives the necessity.
Conversely, suppose $G$ is $d$-sparse and let $u$ be a vertex of minimum degree. Suppose that $G-u$ is not 2-connected. Then there exists a 2-vertex-separation $\{u,w\}\subset V$ in $G$. Let $V=V_1\cup V_2$ where $V_1\cap V_2=\{u,w\}$ and $G[V_i]$ is connected for $i=1,2$. Put $G_i=G[V_i]+uw$ for $i=1,2$. Since $d_G(u)\leq d+1$ and $\{u,w\}$ is a 2-vertex-separation of $G$, we have $d_{G_i}(u)\leq d+1$. 
If $G[V_i-u]$ is 2-connected, then since $G[V_i-u]$ is $d$-sparse and has minimum degree at most $d+1$, at most one arbitrary degree vertex $v$ and $d_{(G[V_i-u])}(t)\leq d+2$ for all $t\in V_i\setminus \{u,v\}$, by induction, $G[V_i-u]$ is $\mathcal{R}_d$-independent for $i=1,2$. If $G[V_i-u]$ is not 2-connected then $G[V_i-u]$ is constructed by gluing 2-connected components at a vertex. Each of these such components fulfills the degree conditions above and thus is $\mathcal{R}_d$-independent by induction. Hence, by Lemma \ref{lem:intbridge}(\ref{it:intbridge:indep}), $G[V_i-u]$ is $\mathcal{R}_d$-independent for $i=1,2$.
Since $u$ has degree at most $d+1$ we can use Lemma \ref{lem:01ext} to show that $G_i$ is $\mathcal{R}_d$-independent for $i=1,2$ unless $u$ has degree exactly $d+1$ and $u$ has exactly 1 neighbour in $V_i-u$ for some $1\leq i \leq 2$. In the former case, let $G'=G+uw$. Since $G_1\cap G_2=K_2$, which is $\mathcal{R}_d$-rigid, Lemma \ref{lem:intbridge}(\ref{it:intbridge:indep}) implies that $G'$ and hence $G$ are $\mathcal{R}_d$-independent. So we may now assume that $u$ has degree exactly $d+1$ and $u$ has exactly 1 neighbour in $V_2-u$. Then we can apply a 1-reduction at $u$, adding an edge $xy$ with $x\in N(u)\cap V_1-u$ and $y\in N(u) \cap V_2-u$ to form a new graph $H$. Lemma \ref{lem:intbridge}(\ref{it:intbridge:rank}) and induction imply that $G-u+xy$ is $\mathcal{R}_d$-independent and Lemma \ref{lem:01ext} implies that $G$ is $\mathcal{R}_d$-independent.

Hence we may assume $G-u$ is 2-connected. Since $G$ is 2-connected, $d(u)\geq 2$ and so $G-u$ has minimum degree at most $d+1$. Since $G-u$ is evidently $d$-sparse, if $d(u)\leq d$ then it follows by induction that $G-u$ is $\mathcal{R}_d$-independent and hence from Lemma \ref{lem:01ext} that $G$ is $\mathcal{R}_d$-independent. Thus $d(u)=d+1$. By the hypotheses, $|\{x\in N(u):d_G(x)\geq d+3\}|\leq 1$. Hence Lemma \ref{lem:reduclem} implies that there exist $x,y\in N(u)$ such that $G-u+xy$ is $d$-sparse. Since the degree hypotheses hold for $G-u+xy$, by induction $G-u+xy$ is $\mathcal{R}_d$-independent and hence Lemma \ref{lem:01ext} implies that $G$ is $\mathcal{R}_d$-independent.
\end{proof}

We conclude this section with two more results that utilise Lemma \ref{lem:intbridge}. The first follows from Theorem \ref{thm:bounded}.

\begin{corollary}\label{cor:circuit}
Let $G$ be connected, $(d+2)$-regular, $d$-sparse and not $\mathcal{R}_d$-independent. Then $G$ is a $\mathcal{R}_d$-circuit. 
\end{corollary}

\begin{proof}
Let $e=uv \in E(G)$. Suppose $G-e$ is disconnected and let $H_1,H_2$ be the connected components of $G-e$. Since $G$ is $d$-sparse and $(d+2)$-regular, $H_1$ and $H_2$ are $d$-sparse with minimum degree at most $d+1$ and maximum degree at most $d+2$. Theorem \ref{thm:bounded} now implies that $H_1$ and $H_2$ are $\mathcal{R}_d$-independent. Then Lemma \ref{lem:intbridge} (\ref{it:intbridge:indep}) implies that $G$ is $\mathcal{R}_d$-independent. Hence we may suppose that $G-e$ is connected for all $e\in E(G)$. Then $G-e$ is $\mathcal{R}_d$-independent by Theorem \ref{thm:bounded}, so $G$ contains a unique $\mathcal{R}_d$-circuit $C$ and $C$ is connected. Suppose $C \subsetneq G$. Then $C$ has maximum degree at most $d+2$ and minimum degree at most $d+1$ and $C$ is $d$-sparse. This contradicts Theorem \ref{thm:bounded}. Hence $G$ is a $\mathcal{R}_d$-circuit.
\end{proof}

Theorem \ref{thm:main} will imply that no graph satisfies the hypothesis of the corollary when $d=3$. We have already observed that when $d\geq 4$ the corresponding class of graphs is known to be non-empty.

Let $G=(V,E)$ and take $X,Y\subset V$. We say that $d(X,Y)$ denotes the number of edges $e\in E$ with $e=xy$, $x\in X\setminus Y$ and $y\in Y\setminus X$. 

\begin{lemma}\label{lem:connected}
Let $G=(V,E)$ be a $\mathcal{R}_d$-circuit. Then $G-e$ is 2-connected for all $e\in E$.
\end{lemma}

\begin{proof}
Suppose first that $G$ is not 2-connected. 
Since $G$ is a $\mathcal{R}_d$-circuit it is connected, so we may write $G=G_1\cup G_2$ where $G_1\cap G_2=K_1$. Lemma \ref{lem:intbridge} (\ref{it:intbridge:indep}) now contradicts the hypothesis that $G$ is a $\mathcal{R}_d$-circuit.
Hence $G$ is 2-connected. 
Suppose there exists $e\in E$ so that $G-e$ is not 2-connected. Then $G-e$ contains a cut-vertex $x$ and $G-e=G_1\cup G_2$ where $V(G_1)\cap V(G_2)=\{x\}$ and $d(V(G_1),V(G_2))=0$. Lemma \ref{lem:intbridge} part \ref{it:intbridge:rank} implies that $r_d(G-e+e)=r_d(G-e)+1$, contradicting the fact that $G$ is a $\mathcal{R}_d$-circuit.
\end{proof}

\section{Contractions in $3$-sparse graphs}
\label{sec:contract}

For the rest of the paper we focus on the case when $d=3$.

Given a 5-regular, $3$-sparse graph we now seek to determine when we can apply an edge contraction or spider contraction. For brevity we will say a \emph{vertex-pair} is a pair of vertices that are non-adjacent.
We say that an edge (resp. vertex-pair) is \emph{reducible} if it is in at least one and at most two triangles (resp. has at least one common neighbour and at most three common neighbours). Secondly, in a 3-sparse graph, we say that a reducible edge (resp. vertex-pair) is \emph{admissible} if the result of an edge contraction (resp. spider contraction) is 3-sparse.

\subsection{Reducibility}

Let $K_{6,6}^-$ denote the unique graph obtained from $K_{6,6}$ by deleting a perfect matching. Note that this graph is 5-regular, 3-sparse, triangle-free and every vertex-pair has either no common neighbours or four common neighbours. The next lemma shows that this graph is unique in this sense.

For a pair of vertices $a,b$ let $N_{a,b}:=N(a)\cap N(b)$ be the set of common neighbours.

\begin{lemma}\label{lem:reducible}
Suppose $G$ is connected, 5-regular, 3-sparse and distinct from $K_{6,6}^-$. Then $G$ has a reducible edge or vertex-pair.
\end{lemma}

\begin{proof}
Let $G$ be a connected, 3-sparse, 5-regular graph and suppose that $G$ does not have a reducible edge or vertex-pair.
Assume every edge contained in a triangle is in at least three triangles.
Suppose first that $G$ has an edge $e=uv$ contained in a triangle. Then $e$ is contained in at least three triangles, say $\{u,v,a\}$, $\{u,v,b\}$ and $\{u,v,c\}$. Let $d$ denote the final neighbour of $u$. 
Since $f=au$ is an edge in a triangle and $d(u)=5$, we may assume $ab \in E$. The edge $g=uc$ is contained in a triangle, thus must add at least two of $ac$, $bc$, $dc$ so, without loss of generality, we may assume $ac \in E$. Since $K_5$ is not 3-sparse, $bc \notin E$. Therefore $cd \in E$ (otherwise $g$ is reducible). 

Since $d(c)=5$, there is $w \in V$ such that $cw \in E$. The edge $h=cd$ is in a triangle, which implies two of $ad$, $vd$, $dw$ are in $E$. Without loss of generality $vd \in E$. Again since $K_5$ is not 3-sparse we have $ad \notin E$ and thus $dw \in E$. Since $d(u)=5$, $\{u,w\}$ is a vertex-pair with $c,d \in N_{u,w}$ we must have $|N_{u,w}| \geq 4$. Since $N(u)=\{a,b,c,d,v\}$ and $vw\notin E$, we have $wa, wb \in E$.
However $i(\{a,b,c,d,u,v,w\})\geq 16>3 \cdot 7-6 = 15 $, contradicting the assumption that $G$ is 3-sparse.

Hence we may assume $G$ is triangle free. If $G$ has no 4-cycles then, since $G$ is connected, it is easy to see that there exists a vertex-pair with exactly one common neighbour.
Thus we may assume $G$ has a 4-cycle. This implies, since $G$ is triangle free, that we have a vertex-pair $\{u,v\}$ with at least two common neighbours. 

Assume now that every vertex-pair with a common neighbour has exactly 5 common neighbours.
Let $\{u,v\}$ be a vertex-pair with exactly five common neighbours, say $N_{u,v} = \{a,b,c,d,e\}$. 
Since $G$ is triangle free, $\{a,b\}$ is a vertex-pair 
and since $u,v \in N_{a,b}$ we have $w,x,y\in V$ with $aw,bw,ax,bx,ay,by \in E$.
Notice that $\{a,c\}$ is a vertex-pair and $u,v\in N_{a,c}$. since $d(a)=5$ we have $cw, cx, cy \in E$. Similarly for the vertex-pairs $\{a,d\}$ and $\{a,e\}$. This gives a subgraph isomorphic to $K_{5,5}$ contrary to 3-sparsity.
   
Hence we may assume that there exists a vertex-pair $\{u,v\}$ with exactly four common neighbours, $N_{u,v}=\{a,b,c,d\}$.
Suppose every other vertex-pair has at least four common neighbours. Since $G$ is 5-regular, there exists $w,x \in V$ such that $uw, vx \in E$. Consider the vertex-pair $\{a,b\}$. We have $u,v \in N_{a,b}$ and since $G$ is triangle-free there exists $y,z\in V$ such that $ay, by, az, bz \in E$. 
The vertex-pair $\{a,c\}$ has two common neighbours so without loss of generality we may assume $yc \in E$. 
Next the vertex-pair $\{a,x\}$ has common neighbour $v$. Since $ux\notin E$, there exists $p\in V$ such that $ap,xp \in E$ and $xy,xz \in E$. 
Similarly the vertex-pair $\{a,w\}$ has a common neighbour and $wv\notin E$, thus $wy, wz, wp \in E$. 
Next the vertex-pair $\{a,d\}$ has common neighbours $u,v$ and hence we have $dz,dp \in E$. 
The vertex-pair $\{c,d\}$ also has common neighbours $u,v$. 
Since $yd,zc\notin E$ we have $pc\in E$ and there exists $t\in V$ such that $tc,td \in E$. 
Now $\{u,t\}$ is a vertex-pair with $c,d\in N_{u,t}$. Since $at\notin E$ we have $bt,wt\in E$.

Finally consider the vertex-pair $\{d,x\}$ and note that $p,v,z \in N_{d,x}$. Since $dy,xu\notin E$ and $\{d,x\}$ has at least four common neighbours we have $xt\in E$. Now the partition $\{\{a,b,c,d,w,x \},\\ \{u,v,p,y,z,t\}\}$ and the missing edges $at,bp,cz,dy,wv,xu$ together with the hypothesis that $G$ is 5-regular implies that $G=K_{6,6}^-$, contradicting our assumption.
\end{proof}

In Subsection \ref{subsec:structure} we will see that not all reducible edges or vertex-pairs are admissible. Hence we next establish a variant of Lemma \ref{lem:reducible} which uses the following concept. 
We say that a subset $X\subseteq V$ is a \emph{core} of $G$ if $G[X]$ has minimum degree at least 3 and every vertex $w\in V \setminus X$ has at most two neighbours in $X$. 
When $X\subsetneq V$, we say $X$ is a \emph{proper core}.

\begin{lemma}\label{lem:outside2}
Suppose $G=(V,E)$ is connected, 5-regular, 3-sparse and distinct from
$K_{6,6}^-$. 
Let $X\subsetneq V$ be a proper core. 
Then there exists a reducible edge $xy$ or a reducible vertex-pair $\{x,y\}$ with $x\in X$ and $y\in V\setminus X$.
\end{lemma}

\begin{proof}
Suppose $uv$ is an edge with $u\in V\setminus X$ and $v\in X$.
Since $X$ is a core, $v$ (resp. $u$) has at least three neighbours in $X$ (resp $V/X$), and by 5-regularity at most 2 neighbours in $V/X$ (resp. $X$).
It is now easy to see that if $uv$ is in at least three triangles we obtain a contradiction to 5-regularity. Hence we may assume that $uv$ is in no triangles of $G$ (otherwise $uv$ is reducible).  By the same argument every edge with one end-vertex in $X$ and one in $V\setminus X$ is contained in no triangles of $G$.

Let $w\in X$ be a neighbour of $v$. Now $\{u,w\}$ is a vertex-pair with $v\in N_{u,w}$. Hence if $\{u,w\}$ is not reducible then they have at least four common neighbours. Since $u$ has at most one neighbour in $X-v$ and $w$ has at most two neighbours in $V\setminus X$ we have $v,y,s,t \in N_{u,w}$ where $y\in X$ and $s,t\in V \setminus X$.
Let $r$ be the unique element in $(N(w)\setminus \{v,y\})\cap X$. Since $sw$ is in no triangles, $\{r,s\}$ is a vertex-pair with a common neighbour $w$. Hence $\{r,s\}$ has at least four common neighbours.
Since $s$ has at most one neighbour in $X-w$ and $r$ has at most two neighbours in $V \setminus X$ we have $w,a,b,c \in N_{r,s}$ where $a\in X$ and $b,c\in V\setminus X$ (note that if $ru\in E$ then $X$ is not a proper core and if $rt\in E$ then $r,t,w$ induce a triangle). See Figure \ref{fig:lem35}(a).

\begin{figure}[ht]
\begin{center}
\begin{tikzpicture}[scale=0.5]

\filldraw (-10,-2) circle (2pt) node[anchor=east]{$w$};
\filldraw (-10,-4) circle (2pt) node[anchor=east]{$a$};
\filldraw (-10,2) circle (2pt) node[anchor=east]{$u$};
\filldraw (-10,4) circle (2pt) node[anchor=east]{$c$};
\filldraw (-10,6) circle (2pt) node[anchor=east]{$b$};

\filldraw (-6,-2) circle (2pt) node[anchor=west]{$v$};
\filldraw (-6,-4) circle (2pt) node[anchor=west]{$y$};
\filldraw (-6,-6) circle (2pt) node[anchor=west]{$r$};
\filldraw (-6,2) circle (2pt) node[anchor=west]{$t$};
\filldraw (-6,4) circle (2pt) node[anchor=west]{$s$};

\draw (-8,-4) circle (125pt); 
\filldraw (-8,-7) circle (0pt) node[anchor=north]{$X$};
\filldraw (-8,-9) circle (0pt) node[anchor=north]{(a)};

\draw[black]
(-6,-2) -- (-10,-2);

\draw[black]
(-6,-2) -- (-10,2);

\draw[black]
(-6,-4) -- (-10,2);

\draw[black]
(-6,-4) -- (-10,-2);

\draw[black]
(-6,-6) -- (-10,-4);

\draw[black]
(-6,-6) -- (-10,-2);

\draw[black]
(-6,-6) -- (-10,4);

\draw[black]
(-6,-6) -- (-10,6);

\draw[black]
(-6,2) -- (-10,2);

\draw[black]
(-6,2) -- (-10,-2);

\draw[black]
(-6,4) -- (-10,6);

\draw[black]
(-6,4) -- (-10,4);

\draw[black]
(-6,4) -- (-10,2);

\draw[black]
(-6,4) -- (-10,-2);

\draw[black]
(-6,4) -- (-10,-4);

\filldraw (3,-2) circle (2pt) node[anchor=east]{$w$};
\filldraw (3,-4) circle (2pt) node[anchor=east]{$a$};
\filldraw (3,-6) circle (2pt) node[anchor=east]{$e$};
\filldraw (3,2) circle (2pt) node[anchor=east]{$u$};
\filldraw (3,4) circle (2pt) node[anchor=east]{$c$};
\filldraw (3,6) circle (2pt) node[anchor=east]{$b$};

\filldraw (7,-2) circle (2pt) node[anchor=west]{$v$};
\filldraw (7,-4) circle (2pt) node[anchor=west]{$y$};
\filldraw (7,-6) circle (2pt) node[anchor=west]{$r$};
\filldraw (7,2) circle (2pt) node[anchor=west]{$t$};
\filldraw (7,4) circle (2pt) node[anchor=west]{$s$};
\filldraw (7,6) circle (2pt) node[anchor=west]{$d$};

\draw (5,-4) circle (125pt); 
\filldraw (5,-7) circle (0pt) node[anchor=north]{$X$};
\filldraw (5,-9) circle (0pt) node[anchor=north]{(b)};

\draw[black]
(7,-2) -- (3,-2);

\draw[black]
(7,-2) -- (3,4);

\draw[black]
(7,-2) -- (3,2);

\draw[black]
(7,-2) -- (3,-4);

\draw[black]
(7,-2) -- (3,-6);

\draw[black]
(7,-4) -- (3,6);

\draw[black]
(7,-4) -- (3,2);

\draw[black]
(7,-4) -- (3,-2);

\draw[black]
(7,-4) -- (3,-4);

\draw[black]
(7,-4) -- (3,-6);

\draw[black]
(7,-6) -- (3,-6);

\draw[black]
(7,-6) -- (3,-4);

\draw[black]
(7,-6) -- (3,-2);

\draw[black]
(7,-6) -- (3,4);

\draw[black]
(7,-6) -- (3,6);

\draw[black]
(7,2) -- (3,6);

\draw[black]
(7,2) -- (3,4);

\draw[black]
(7,2) -- (3,2);

\draw[black]
(7,2) -- (3,-2);

\draw[black]
(7,2) -- (3,-6);

\draw[black]
(7,4) -- (3,6);

\draw[black]
(7,4) -- (3,4);

\draw[black]
(7,4) -- (3,2);

\draw[black]
(7,4) -- (3,-2);

\draw[black]
(7,4) -- (3,-4);

\draw[black]
(7,6) -- (3,6);

\draw[black]
(7,6) -- (3,4);

\draw[black]
(7,6) -- (3,2);

\draw[black]
(7,6) -- (3,-4);

\draw[black]
(7,6) -- (3,-6);
\end{tikzpicture}  
\caption{An illustration of the proof of Lemma \ref{lem:outside2}. In (b) we find $K_{6,6}^-$ and $G[X]=K_{3,3}$. In fact any proper core in $K_{6,6}^-$ induces a copy of $K_{3,3}$.} \label{fig:lem35}
\end{center}
\end{figure}
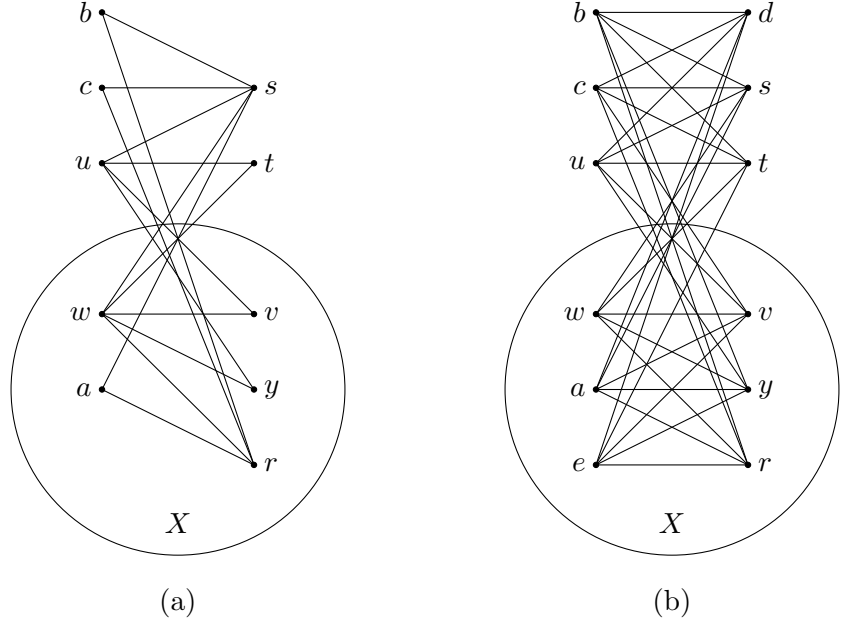

Now $\{w,b\}$ is a vertex-pair with two common neighbours and $b$ can be adjacent to at most one of $v,y$, which are equivalent, thus $by,bt\in E$. Similarly $\{w,c\}$ is a vertex-pair with two common neighbours and $ct\in E$. Note $ab \in E$ would induce a triangle, hence $\{a,b\}$ is a vertex-pair. We obtain contradictions if $b$ has another neighbour in $X$ or if either $bc \in E$ or $bu\in E$, therefore let $d \in V\setminus X$ with $db \in E$. Now $\{s,y\}$ is a vertex-pair with $N_{s,y}\supset \{b,u,w\}$. Since $y$ already has two neighbours in $V\setminus X$, its remaining two neighbours must be in $X$. Thus, for $\{s,y\}$ to have four common neighbours, we have $ay\in E$. Triangles are created if $yr$ or $yv$ exists, so we must add a new vertex $e\in X$ such that $ye\in E$. We have another vertex-pair $\{y,d\}$ and both $y$ and $w$ have degree 5, thus $da,du,de\in E$. If $rd\in E$ then we have a triangle so $\{r,d\}$ is a vertex-pair with $N_{r,d}\supset \{a,b\}$. Since $r$ needs one more neighbour in $X$ and $d$ cannot have another neighbour in $X$, we have $dc,re\in E$. 

Hence $\{c,e\}$ is a vertex-pair so $et\in E$ and $e$ needs another neighbour in $X$. If this is a vertex not previously labelled, say $f$, then $\{f,t\}$ is a vertex-pair but $t$ and $f$ cannot have four common neighbours.
So $cv,ev\in E$ since $ca\in E$ creates a triangle. 
The only remaining vertices without degree 5 are $a$ and $v$ and the vertex-pair $\{s,v\}$ gives that $av\in E$. See Figure \ref{fig:lem35}(b).
Thus $G$ is the graph $K_{6,6}^-$ with the partition $\{d,s,y,v,t,r\}$ and $\{e,a,b,c,u,w\}$ and $X=\{v,w,r,y,a,e\}$. 
\end{proof}

\subsection{Special sets in 5-regular graphs}

First we consider special sets which arise as `blockers' to our contraction operations.
Let $G=(V,E)$ and take $X\subset V$. If $i(X)=3|X|-k$ then we will call $X$ a \emph{$k$-set}. Note that in a 3-sparse graph a $6$-set induces a 3-tight subgraph and hence $k\geq 6$.

\begin{lemma}\label{lem:small}
    Let $G=(V,E)$ be a connected, 3-sparse, 5-regular graph. Let $X\subseteq V$ be a $k$-set. Then $|X|\leq 2k$ with equality if and only if $X=V$. Moreover $d(X,V \setminus X)=2k-|X|$.
\end{lemma}

\begin{proof}
    Since $G$ is connected and 5-regular, and $X\subseteq V$, we have $i(X)\leq \frac{5}{2}|X|$. Since $X$ is a $k$-set it follows that $6|X|-2k\leq 5|X|$. Moreover if $X\subsetneq V$ then $i(X)< \frac{5}{2}|X|$ and the same argument completes the first conclusion. For the second, 
    $$5|X|=2i(X)+d(X,V \setminus X)=6|X|-2k+d(X,V \setminus X).$$
    Thus $d(X,V \setminus X)=2k-|X|$.
\end{proof}

It follows immediately from Lemma \ref{lem:small} that if $|X|\geq k$, then $d(X,V \setminus X)\leq k$.

\begin{lemma}\label{lem:uniondeg}
Let $G=(V,E)$ be 3-sparse and take $X\subset V$ with $|X|\geq 4$, such that $X$ is a $k$-set for some $k\in \{6,7,8,9\}$. Then $d_{G[X]}(v)\geq 9-k$ for all $v\in X$. 
\end{lemma}

\begin{proof}
Suppose $d_{G[X]}(v)< 9-k$. Since $X$ is a $k$-set we have $i(X-v)>3|X|-k -(9-k)=3|X-v|-6$, contradicting the 3-sparsity of $G$ unless $|X-v|=2$ (which contradicts the hypothesis that $|X|\geq 4$).
\end{proof}

\begin{lemma}\label{lem:7sets}
Let $G=(V,E)$ be 5-regular and 3-sparse and suppose the only 6-sets in $G$ induce copies of $K_3$.
Then every 7-set $X\subset V$, with $|X|\geq 5$ has $d_{G[X]}(v)\geq 3$ for all $v\in X$. 
\end{lemma}

\begin{proof}
Suppose $v\in X$ has $d_{G[X]}(v)<3$ then either we contradict the 3-sparsity of $G$ or $v$ has exactly two neighbours in $X$ and $X-v$ is a 6-set on at least 4 vertices.
\end{proof}

\begin{lemma}\label{lem:no_triangle_intersect}
Let $G=(V,E)$ be connected, 5-regular, 3-sparse and suppose that every 6-set in $G$ induces a copy of $K_3$. Suppose $X,Y\subset V$ are 7- or 8-sets on at least 6 vertices such that $G[X\cap Y]$ is isomorphic to $K_3$. Then $G$ contains a subgraph isomorphic to $K_4-e$.
\end{lemma}

\begin{proof}
Let $T=X\cap Y$. Then Lemma \ref{lem:small} implies that $d(T,V\setminus T)=9$. Hence, for some $Z\in \{X,Y\}$, we have $d(T,Z\setminus T)\leq 4$.
If $Z\setminus T$ is not a 6-set then
$$i(Z) = i(Z\setminus T)+i(T)+d(T,Z\setminus T)\leq 3|Z\setminus T|-7+3+4=3|Z|-9,$$
which contradicts the hypotheses on $X$ and $Y$.
Hence $Z\setminus T$ is a 6-set inducing a copy of $K_3$ and $|Z|=6$. Now
$$i(Z) = i(Z\setminus T)+i(T)+d(T,Z\setminus T)=6+d(T,Z\setminus T)\leq 10.$$
Hence $Z$ is an 8-set and
$d(T,Z\setminus T)=4$. Thus $G[Z]$ consists of two vertex-disjoint triangles joined by 4 edges. The conclusion follows.
\end{proof}

Let $G=(V,E)$ be a connected graph and take $k\in \mathbb{N}$. A \emph{$k$-edge cutset} is a subset $S\subset E$ of size $|S|=k$ such that $G-S$ is disconnected. A $k$-edge cutset is \emph{trivial} if one component of $G-S$ is isomorphic to $K_1$. $G$ is \emph{essentially $k$-edge connected} if the only $j$-edge cutsets with $j<k$ are trivial.

\begin{lemma}\label{lem:ess-ec}
Let $G=(V,E)$ be a 3-sparse graph with $X\subset V$ and $|X|\geq 5$. Suppose that there is no 6-set in $G$. Moreover suppose that $i(X)=3|X|-k$ with $7 \leq k\leq 11$. Then $G[X]$ is essentially $(12-k)$-edge-connected. 
\end{lemma}

\begin{proof}
Suppose, for a contradiction, that there exists $S\subset E(G[X])$ with $|S|\leq 11-k$ such that $G[X]-S$ has two components $G[X_1]$ and $G[X_2]$, and $|X_i|>1$ for $i=1,2$.
By 3-sparsity we have $i(X_i)\leq 3|X_i|-5$ with equality if and only if $|X_i|=2$ for $i=1,2$ and if $|X_i|\geq 3$ then we have $i(X_i)\leq 3|X_i|-7$.
If $|X_1|=|X_2|=2$, then $|X|=4$, contradicting our assumption.
Hence suppose without loss of generality that $|X_2|\geq 3$. 
Then
\begin{eqnarray*}
        i(X) &=& i(X_1)+i(X_2)+d(X_1,X_2)\\
        &\leq& 3|X_1|+3|X_2|-12+11-k\\
        &=& 3|X|-1-k,
\end{eqnarray*}
a contradiction. 
\end{proof}

\begin{lemma}\label{lem:8set}
Let $G=(V,E)$ be a 5-regular, 3-sparse graph with no 6-set. For $Y,Z\subset V$ with $|Y|,|Z|\geq 6$, suppose $Y$ is a $j$-set and $Z$ is a $k$-set for $j,k\in \{7,8\}$.
If $Y$ and $Z$ both have minimum degree 3 and $Y\cap Z=\{a,b\}$ with $ab\in E$, then $Y$ and $Z$ are both 8-sets.
\end{lemma}

\begin{proof}
Suppose $X\in \{Y,Z\}$ is a 7-set. Then $a,b$ both have degree $3$ in $G[X]$ and thus  $X\setminus \{a,b\}$ is a 6-set. 
This contradicts the hypothesis that $G$ has no 6-set.
\end{proof}

We will also need the following elementary observation which follows from Mantel's theorem.

\begin{lemma}\label{lem:trifree}
Let $G=(V,E)$ be a graph and let $X\subset V$ be a 7-set with $|X|=n$ and $Y\subset V$ be an 8-set with $|Y|=n'$. If $G[X]$ is triangle-free then $n\leq 3$ or $n\geq 9$ and if $G[Y]$ is triangle-free then $n'\leq 4$ or $n'\geq 8$.   
\end{lemma}

\subsection{Intersections and unions}

\begin{lemma}\label{lem:unionof66}
Let $G=(V,E)$ be 3-sparse, 5-regular and take $X,Y\subset V$ with $|X|,|Y|\geq 4$ such that $X\cap Y \neq \emptyset$ and $X$, $Y$ are $6$-sets. Then $|X\cap Y|\geq 3$, $X\cup Y$ and $X\cap Y$ are both 6-sets and $d(X,Y)=0$.
\end{lemma}

\begin{proof}
If $|X\cap Y|=1$ or $X\cap Y=\{a,b\}$ and $ab\notin E$ then Lemma \ref{lem:uniondeg} contradicts the fact that $G$ is 5-regular. If $X\cap Y=\{a,b\}$ and $ab\in E$ then Lemma \ref{lem:uniondeg} implies that $d_{G[X]}(v)=3=d_{G[Y]}(v)$ for each $v\in \{a,b\}$.
It follows that 
$$i(X\setminus \{a,b\})=i(X)-5=3|X|-6-5=3|X\setminus \{a,b\}|-5,$$ contradicting the 3-sparsity of $G$.
Hence $|X\cap Y|\geq 3$ and we have 
\begin{eqnarray*}
3|X\cup Y|-6 + 3|X\cap Y|-6 &\geq& i(X\cup Y)+i(X\cap Y) \\ &=& i(X)+i(Y)+d(X,Y) \\ &=& 3|X|-6+3|Y|-6+d(X,Y) \\ &=& 3|X\cup Y|+3|X\cap Y|-12+d(X,Y).
\end{eqnarray*}
If $|X\cap Y|\geq 3$ then $d(X,Y)=0$, equality holds throughout and $X\cup Y$ and $X\cap Y$ are both 6-sets.
\end{proof}

\begin{lemma}\label{lem:unionof77}
    Let $G=(V,E)$ be 3-sparse, 5-regular and take $X,Y \subset V$ such that $X \cap Y \neq \emptyset$ and $X$, $Y$ are both $7$-sets. 
    \begin{enumerate}
        \item If $X\cap Y = \{a,b\}$ with $ab \in E$, then $d(X,Y) \leq 3$ and $i(X\cup Y)= 3|X\cup Y|-9+d(X,Y)$, and
        \item if $|X\cap Y|\geq 3$, then $d(X,Y) \leq 2$, $X\cup Y$ is an $m$-set and $X \cap Y$ is an $n$-set where $m,n \geq 6$ and such that $m+n=14-d(X,Y)$.
    \end{enumerate}
Moreover, if $G[X]$ and $G[Y]$ have minimum degree 3 then either 1 or 2 holds.
\end{lemma}

\begin{proof}
If $X\cap Y=\{a,b\}$ with $ab\in E$, then we have
    \begin{eqnarray*}
        i(X\cup Y) &=& i(X)+i(Y)+d(X,Y)-i(X\cap Y)\\
        &=& 3|X|-7+3|Y|-7+d(X,Y)-1\\
        &=& 3|X\cup Y|+3|X\cap Y|-14-1 +d(X,Y)\\
        &=& 3|X\cup Y|+6-15+d(X,Y)\\
        &=& 3|X\cup Y|-9+d(X,Y).
    \end{eqnarray*}
    By the 3-sparsity of $G$, we have $9-d(X,Y)\geq 6$ and $d(X,Y)\leq 3$. 
    
    If $|X\cap Y|\geq 3$ we have
    \begin{eqnarray*}
    3|X\cup Y|-6 + 3|X\cap Y|-6 &\geq& 
    i(X\cup Y)+i(X\cap Y) \\ &=& i(X)+i(Y)+d(X,Y) \\ &=& 
    3|X|-7+3|Y|-7+d(X,Y) \\ &=& 3|X\cup Y|+3|X\cap Y|-14+d(X,Y).
    \end{eqnarray*}
    Thus $d(X,Y)\leq 2$ and we can see that $X\cup Y$ is an $m$-set and $X \cap Y$ is an $n$-set $m,n \geq 6$ such that $m+n=14-d(X,Y)$. 

The final conclusion follows, because if $|X\cap Y|=1$ or $X\cap Y=\{a,b\}$ with $ab\notin E$ then we have a contradiction to the 5-regularity of $G$. 
\end{proof}

\begin{lemma}\label{lem:unionof88}
    Let $G=(V,E)$ be 3-sparse, 5-regular and take $X,Y \subset V$ such that $X \cap Y \neq \emptyset$ and $X$, $Y$ are both $8$-sets. 
    \begin{enumerate}
        \item If $X\cap Y = \{a,b\}$ with $ab \in E$, $d(X,Y) \leq 5$ and $i(X\cup Y)= 3|X\cup Y|-11+d(X,Y)$, and
        \item if $|X\cap Y|\geq 3$, $d(X,Y) \leq 4$ and $X\cup Y$ is an $m$-set, $X \cap Y$ is an $n$-set where $m,n \geq 6$ such that $m+n=16-d(X,Y)$.
    \end{enumerate}
    Moreover, if $G[X]$ and $G[Y]$ have minimum degree 3 then either 1 or 2 holds.
\end{lemma}

In the proof, and in the proof of the lemma that follows, the final conclusion has an identical proof to the previous result and hence is omitted.

\begin{proof}
    If $X\cap Y=\{a,b\}$ with $ab\in E$, then we have $i(X\cup Y)=3|X\cup Y|-11+d(X,Y).$
    By the 3-sparsity of $G$, we have $11-d(X,Y)\geq 6$ and hence $d(X,Y)\leq 5$. 
    If $|X\cap Y|\geq 3$ we have $3|X\cup Y|-6 + 3|X\cap Y|-6 \geq 3|X\cup Y|+3|X\cap Y|-16+d(X,Y).$
    Thus $d(X,Y)\leq 4$ and we can see that $X\cup Y$ is an $m$-set and $X \cap Y$ is an $n$-set where $m,n \geq 6$ such that $m+n=16-d(X,Y)$. 
\end{proof}

\begin{lemma}\label{lem:unionof78}
Let $G=(V,E)$ be 3-sparse, 5-regular and take $X,Y\subset V$ such that $X\cap Y \neq \emptyset$, $X$ is a $7$-set and $Y$ is an $8$-set.
\begin{enumerate}
    \item If $X \cap Y = \{a,b\}$ with $ab \in E$, then $d(X,Y)\leq 4$ and $X\cup Y$ is a $(10-d(X,Y))$-set, and
    \item if $|X\cap Y|\geq 3$, then $d(X,Y)\leq 3$ and $X\cup Y$ is an $m$-set, $X\cap Y$ is an $n$-set where $m,n\geq 6$ such that $m+n=15-d(X,Y)$.
\end{enumerate}
Moreover, if $G[X]$ and $G[Y]$ have minimum degree 3 then either 1 or 2 holds.
\end{lemma}

\begin{proof}
    If $X\cap Y=\{a,b\}$ with $ab\in E$, then we have $i(X\cup Y)=3|X\cup Y|-10+d(X,Y).$
This implies that $X \cup Y$ is a $(10-d(X,Y))$-set and, by 3-sparsity of $G$, $d(X,Y)\leq 4$. 
    If $|X\cap Y|\geq 3$ we have $3|X\cup Y|-6 + 3|X\cap Y|-6 \geq 3|X\cup Y|+3|X\cap Y|-15+d(X,Y).$
Thus $d(X,Y)\leq 3$ and $X\cup Y$ is an $m$-set, $X\cap Y$ is an $n$-set $n,m\geq 6$ such that $n+m=15-d(X,Y)$.
\end{proof}

\subsection{Admissible reductions}
\label{subsec:structure}

\begin{lemma}\label{lem:nonadmedge}
Let $G=(V,E)$ be a 3-sparse graph containing a reducible edge $e=uv$ and let $X \subset V$ such that $u,v \in X$. Then:
    \begin{enumerate}
        \item if $e$ is in exactly one triangle $u,v,w$ of $G$ then $e$ is non-admissible if and only if $|X|\geq 6$ and either (a) $X$ is a 6-set or (b)  $w \notin X$ and $X$ is a 7-set.
        \item if $e$ is in exactly two triangles $u,v,w$ and $u,v,x$ of $G$ then $e$ is non-admissible if and only if $|X|\geq 6$ and either (a) at most one of $w,x$ is in $X$ and $X$ is a 6-set or (b) $w,x \notin X$ and $X$ is a 7-set.
    \end{enumerate}
\end{lemma}

\begin{proof}
   In both cases the sufficiency is clear. For the necessity we prove the contrapositive. Let $u'$ be the vertex we obtain by contracting $e=uv$ and let $G'=(V',E')$ be the graph we obtain after the contraction. Let $Y' \subset V'$ and let $Y \subset V$ be the corresponding set in $G$. In each case above we want to show $i(Y')\leq 3|Y'|-6$ for all $Y' \subset V'$ when the respective subcases fail. In both cases if $u' \notin Y'$ then $Y' \subset V$ and $i_{G'}(Y')=i_G(Y)\leq 3|Y|-6$. Thus consider $u' \in Y'$, which implies either $u,v \in Y$ or without loss of generality $u \in Y, v \notin Y$. Let $u,v \in Y$. In both cases of our lemma, by our assumption, either $i(Y)\leq 3|Y|-8$, $i(Y)\leq 3|Y|-7$ or (exclusively in case 2) $i(Y)\leq 3|Y|-6$, which imply in their respective cases that $i(Y')\leq 3|Y'|-6$. Thus let $u\in Y, v \notin Y$. Let $Y$ be an $l$-set and let $v$ have $k$ neighbours $t_1,\dots ,t_k$ in $Y$ unique from $w$ (or $w,x$ in case 2) and non-adjacent to $u$. If $k<l-5$, then we have $i(Y')\leq 3|Y'|-6$, since
    $$i(Y')=3|Y|-l+k=3|Y'|-l+k<3|Y'|-l+l-5=3|Y'|-5.$$ 
    If $k\geq l-5$, then we consider $Y\cup v$. If no other vertex in the triangle(s) is contained in $Y$, we obtain $i_G(Y\cup v) \geq 3|Y\cup v|-7$. If exactly one other vertex in the triangle(s) excluding $v$ is contained in $Y$, we obtain $i_G(Y\cup v) \geq 3|Y\cup v|-6$. Exclusively in case 2, if $w,x\in Y$, we obtain $i_G(Y\cup v) \geq 3|Y\cup v|-5$. Therefore if $k\geq l-5$ we always have a contradiction to our assumptions. Since the 3-sparsity inequality is only required for sets of size at least 3, $|X|\geq 4$. If $|X|=4$ or $5$ then $e$ exceeds the number of triangles allowed in our assumptions. Thus $|X| \geq 6$.
\end{proof} 

\begin{lemma}\label{lem:nonadmpair}
    Let $G=(V,E)$ be a 3-sparse graph containing a reducible vertex-pair $\{u,v\}$ and let $X \subset V$ such that $u,v \in X$. Then:
    \begin{enumerate}
        \item If $\{u,v\}$ has exactly one common neighbour $a$ then $\{u,v\}$ is non-admissible if and only if $|X|\geq 6$ and either (a) $X$ is a 6-set,
        (b) $X$ is a 7-set or (c) $a\notin X$ and $X$ is an 8-set.
        \item If $\{u,v\}$ has exactly two common neighbours $a,b$ then $\{u,v\}$ is non-admissible if and only if $|X|\geq 6$ and either (a) $X$ is a 6-set, (b) at most one of $a,b$ is in $X$ and $X$ is a 7-set or (c) $a,b \notin X$ and $X$ is an 8-set.
        \item If $\{u,v\}$ has exactly three common neighbours $a,b,c$ then $\{u,v\}$ is non-admissible if and only if $|X|\geq 6$ and either (a) $X$ is a 6-set, (b) at most one of $a,b,c$ is in $X$ and $X$ is a 7-set or (c) $a,b,c \notin X$ and $X$ is an 8-set.
    \end{enumerate}
\end{lemma}

\begin{proof}
  In all three cases sufficiency is clear. For the necessity we prove the contrapositive. Let $u'$ be the vertex we obtain by contracting $e=uv$ and let $G'=(V',E')$ be the graph we obtain after the contraction. Let $Y'\subset V'$ with corresponding set $Y \subset V$. In each of the above cases we want to show $i(Y')\leq 3|Y'|-6$ for all $Y' \subset V'$ when the relevant subcases fail. In each of the cases if $u' \notin Y'$ then $Y' \subset V$ and $i_{G'}(Y')=i_G(Y)\leq 3|Y|-6$. Thus consider $u' \in Y'$, which implies either $u,v \in Y$ or without loss of generality $u \in Y,v \notin Y$. Let $u,v \in Y$. In all three cases our assumptions imply $i(Y')\leq 3|Y'|-6$. Thus let $u\in Y, v \notin Y$. Let $Y$ be an $l$-set and let $v$ have neighbours $t_1,\dots,t_k$ in $Y$. If $k<l-5$, then we have $i(Y')\leq 3|Y'|-6$. If $k\geq l-5$, then we consider $Y\cup v$ for each possible $Y$ not containing $v$ and obtain a contradiction to our assumptions. Since the 3-sparsity inequality is only required for sets of size at least 3, $|X|\geq 4$. If $|X|=4$ or $5$ either we contradict that $\{u,v\}$ is a vertex-pair or our assumption on the common neighbours of $\{u,v\}$. Thus $|X| \geq 6$.
\end{proof}

It will sometimes be convenient to refer to one of the sets arising in Lemmas \ref{lem:nonadmedge} and \ref{lem:nonadmpair} without specifying precisely which one. Hence we will say that $X\subset V$ is a \emph{blocker} if it satisfies at least one of the options in those two lemmas.

\section{Main result}
\label{sec:proof}

In this section we put our results together to prove Theorem \ref{thm:main}.
The proof is long so we split it across two subsections. We first prove several technical results about particular cases and then prove the theorem itself.

\subsection{Admissible sub-structures}

We start by ruling out certain sub-structures in the proof of the main theorem. The first is the main step in showing that we do not need to worry about 5-regular, 3-sparse graphs that contain a subgraph isomorphic to $K_4$.

\begin{proposition}\label{prop:k4admiss}
    Let $G=(V,E)$ be a 5-regular, 3-sparse graph in which every 6-set induces a complete graph. Suppose $G$ contains a subgraph isomorphic to $K_4$. Then $G$ has an admissible edge.
\end{proposition}

\begin{proof}
Suppose $X\subset V$ is a 6-set with $G[X]=K_4$. Since every 6-set is complete, no vertex in $V \setminus X$ has more than 2 neighbours in $X$. Let $X=\{a,b,c,d\}$ and suppose for a contradiction that every reducible edge in $G$ is non-admissible.

Suppose there exists a vertex $v$ with 2 neighbours in $X$ and, without loss of generality, put $N_{G[X]}(v)=\{a,b\}$. Let $r$ be the final neighbour of the vertex $a$. Then $av$ is a reducible edge in either one or two triangles.
Suppose $vr \in E$ (i.e. $av$ is in exactly two triangles). By Lemma \ref{lem:nonadmedge}, there exists a 7-set $A$ with $|A|\geq 6$ such that $v,a \in A$ and $b,r\notin A$ and $A$ has minimum degree 3 otherwise there exists a 6-set with at least five vertices contradicting our hypotheses. Thus $c,d\in A$. Now $A\cup b$ is a 6-set with at least five vertices, a contradiction. 

Suppose then that $vr\notin E$ (i.e. $av$ is in exactly one triangle). By Lemma \ref{lem:nonadmedge}, there exists a 7-set $B$ with $|B|\geq 6$ such that $a,v \in B$ and $b\notin B$. If $B$ has minimum degree less than 3 then there exists a 6-set in $B$ on at least five vertices, a contradiction. So we may assume $B$ has minimum degree 3. If $c,d\in B$ then $b$ has four neighbours in $B$ and $B\cup b$ is a 6-set, another contradiction. Hence, without loss of generality, $c\notin B$. Now $d\in B$, otherwise $a$ has at most two neighbours in $B$. Since $a$ has three neighbours in $B$ we must have $r\in B$.

Now consider the final neighbour $s$ of $b$. Suppose $s=r$ and $vs\notin E$. Thus $av$ is a reducible edge in exactly one triangle. By Lemma \ref{lem:nonadmedge}, there exists a 7-set $D$ with $|D|\geq 6$ such that $a,v \in D$ and $b\notin D$. If $D$ is not a proper core then there exists a 6-set in $D$ on at least five vertices, a contradiction. So we may assume $D$ is a proper core and hence at least two of $c,d,s$ are in $D$. Then $b$ has four neighbours in $D$ and $D\cup b$ is a 6-set, a contradiction.

Hence we may suppose $s \neq r$.  By symmetry with the case above when $vr\in E$, we may suppose $vs\notin E$. Again we may obtain a 7-set $F$ containing $v,b,s$ with minimum degree 3, $a\notin F$ and $|\{c,d\}\cap F|=1$. Note also that $s\notin B$ (resp. $r\notin F$) otherwise $b$ (resp. $a$) has four neighbours in $B$ (resp. $F$). 
Suppose that $d\in F$. Since $vd\notin E$, Lemma \ref{lem:unionof77} implies that $|B\cap F|\geq 3$. Since $G[B\cap F]$ is not complete, this implies $B\cap F$ is not a 6-set. Since $ab\in E$, $d(B,F)\geq 1$. Lemma \ref{lem:unionof77} now implies that $B\cup F$ is a 6-set, contradicting the hypothesis that every 6-set induces a complete graph.

So $d\notin F$ and $c\in F$. Then $F\cup \{a,d\}$ is a 7-set, $a,v,d\in B\cap (F\cup \{a,d\})$ and $vd\notin E$. So we may apply the argument from the preceding paragraph.
Hence no vertex in $V \setminus X$ has more than one neighbour in $X$ and every edge in $G[X]$ is a reducible edge in exactly two triangles. First consider the edge $ab$. By Lemma \ref{lem:nonadmedge}, there exists a 7-set $J$ with $|J|\geq 6$ such that $a,b \in J$ and $c,d \notin J$. We obtain similar 7-sets for each edge in $G[X]$. We may assume that each 7-set induces a graph with minimum degree at least 3 otherwise we contradict the hypothesis that every 6-set induces a complete graph. 

Let $K$ be the 7-set with $b,c\in K$ and $a,d \notin K$. Since both $J$ and $K$ have minimum degree at least 3, $|J\cap K|\geq 3$. In addition using the facts that $d(J,K)\geq 1$ and $J\cup K$ has at least five vertices, by Lemma \ref{lem:unionof77}, we deduce that $d(J,K)=1$, $J\cup K$ is a 7-set and $J\cap K$ is a 6-set which induces a copy of $K_3$. Thus each of $a,b,c,d$ has two unique neighbours forming a $K_3$, by our assumption that no vertex in $Z-X$ has more than one neighbour in $X$. 
Let $L$ be the 7-set with $a,d\in L$ and $b,c \notin L$. Since both $L$ and $J\cup K$ have minimum degree at least 3, $|L\cap (J\cup K)|\geq 3$. In addition using the facts that $d(L,J\cup K)\geq 2$ and $L\cup (J\cup K)$ has at least five vertices, by Lemma \ref{lem:unionof77}, we deduce that $d(L,J\cup K)=2$ and $L\cup (J
\cup K)$ is a 6-set contradicting the hypothesis that every 6-set induces a complete graph. 
\end{proof}

We next establish two results that guarantee that certain intersecting blockers imply either $\mathcal{R}_3$-independence or that the graph is small. 

\begin{proposition}\label{prop:smallenough}
Let $G$ be a connected, 5-regular, 3-sparse graph which contains no 6-set. Suppose that every reducible vertex-pair is non-admissible. Let $Y$ be a a proper core which is a $j$-set for $j\in \{7,8\}$ and let $\{u,w\}$ be a reducible vertex-pair with $u\in Y$ and $w\notin Y$. Suppose that $Z$ is a blocker for $\{u,w\}$ and $G[Z]$ has minimum degree 3. If $Y\cup Z=\overline{Y\cup Z}=V$ then either $G$ is $\mathcal{R}_3$-independent or $|V|\leq 14$.   
\end{proposition}

\begin{proof}
Suppose $Y\cup Z=\overline{Y\cup Z}=V$ and $G$ is not $\mathcal{R}_3$-independent. We will show that $|V|\leq 14$. First note that $G$ is an $\mathcal{R}_3$-circuit by Corollary \ref{cor:circuit}. Let $Y\cup Z$ be a $k$-set. Since $G$ contains no 6-set we have $k\geq 7$. 
The conclusion of the proposition follows from Lemma \ref{lem:small} when $k=7$. We complete the proof by showing this equality must hold. Suppose $k>7$.
Since $Y$ is a proper core and $G[Z]$ has minimum degree 3, $|Y\cap Z|\geq 2$. Let $Z$ be a $\ell$-set and $d:=d(Y,Z)$. We split the remainder of the proof into two cases based on $|Y\cap Z|$.\\

\noindent \textbf{Case 1. $Y\cap Z=\{a,b\}$ with $ab\in E$.}\\ 

Note first that if $d=0$ then, since $Y\cup Z=V$, we contradict the assumption that $G$ is an $\mathcal{R}_3$-circuit using Lemma \ref{lem:intbridge} (1).  Hence $d\geq 1$.
Also, since $Y\cap Z$ induces a copy of $K_2$, we can apply Lemma \ref{lem:intbridge} (1) and (2), to contradict the hypothesis that $G$ is an $\mathcal{R}_3$-circuit since $Y\cup Z =V$. Hence we may assume that $d\geq 2$. Furthermore Lemma \ref{lem:8set} implies that $j=8=\ell$.
Lemmas \ref{lem:unionof77}, \ref{lem:unionof88} and \ref{lem:unionof78} and the hypothesis that $G$ contains no 6-set now imply that $j=\ell=8$ and $k=11-d$ where $d\in \{2,3\}$.

Suppose first that $k=9$ and $d=2$. Then Lemma \ref{lem:small} implies that $|V|=18$. However, since $j=8=\ell$, Lemma \ref{lem:trifree} implies that $|V|\geq 20$, a contradiction.
Similarly when $k=8$, Lemma \ref{lem:small} implies that $|V|=16$ and we obtain the same contradiction.  
\\

\noindent \textbf{Case 2. $|Y\cap Z|\geq 3$.}\\

Lemmas \ref{lem:unionof77}, \ref{lem:unionof88} and \ref{lem:unionof78} and the hypothesis that $G$ contains no 6-set imply that either: $\{j,\ell\}=\{7,8\}$, $k=8$ and $d=0$; or $j=\ell=8$ and $k\leq 9-d$ with $d\in \{0,1\}$.
Suppose $k=9$. Then $Y\cap Z$ is a 7-set and $d=0$. Now $\{u,w\}$ must have at least one common neighbour not in $Z$, contradicting our assumptions that $d(Y,Z)=0$ and $Y\cup Z=V$.

Hence $k=8$. Then either: $Y$ and $Z$ are 8-sets, $d=0$ and $Y\cap Z$ is an 8-set; $Y$ and $Z$ are 8-sets, $d=1$ and $Y\cap Z$ is an 7-set; or $Y$ is a 7-set and $Z$ is an 8-set or vice-versa, $d=0$ and $Y\cap Z$ is an 7-set. 
If $Y$ and $Z$ are 8-sets, $d=0$ and $Y\cap Z$ is an 8-set, then we contradict the assumption that $Z$ is blocking a reducible vertex-pair. 
Now suppose $Y$ and $Z$ are 8-sets, $d=1$ and $Y\cap Z$ is an 7-set. Let $Y'=Y\setminus (Y\cap Z)$, then, by Lemma \ref{lem:ess-ec}, $d(Y',Y\cap Z)\geq 4$. Hence either $Y'$ is a copy of $K_2$ and $d(Y',Y\cap Z)= 4$ or $Y'$ is a 7-set and $d(Y',Y\cap Z)= 6$. The former contradicts the 5-regularity of $G$, since $Y\cup Z=V$. Suppose the latter, then, by Lemma \ref{lem:small}, $d(Y\cap Z, V\setminus (Y\cap Z))\leq 11$. Hence $d(Y\cap Z, Z\setminus(Y\cap Z))\leq 5$ and we contradict the 5-regularity of $G$.

Finally suppose $Y$ is a 7-set, $Z$ is an 8-set or vice-versa, $d=0$ and $Y\cap Z$ is an 7-set. Then, if $Z$ is an 8-set, $d\geq 1$, contradicting our assumption. Hence $Z$ is a 7-set and the pair $\{u,w\}$ with $u\in Y\cap Z$ and $w\in Z\setminus Y$ has exactly one common neighbour. By Lemma \ref{lem:small}, $d(Y\cap Z, V\setminus (Y\cap Z))\leq 11$. Let $Y'=Y\setminus (Y\cap Z)$ and $Z'=Z\setminus (Y\cap Z)$. By Claim \ref{claim:6set}, $Y'$ and $Z'$ both have at least two vertices. If one of $Y'$ and $Z'$ is a copy of $K_2$, then we contradict the 5-regularity of $G$. Hence $i(Y')\leq 3|Y'|-7$ and $i(Y')\leq 3|Y'|-7$, however then we contradict the fact that $d(Y\cap Z, V\setminus (Y\cap Z))\leq 11$.
\end{proof}

For the second we will apply the following lemma which will also be utilised in the proof of our main result. For $X\subset V$, let $\overline{X}$ denote the \emph{(core-)closure of $X$}, this is the set obtained from $X$ by recursively adding $w\in V-X$ if $w$ has at least 3 neighbours in $X$ until every vertex of $V-X$ not added to $X$ has at most 2 neighbours in $\overline{X}$. Note that the (core-)closure of a set is a core of the graph if and only if $G[X]$ has minimum degree 3.

\begin{lemma}\label{lem:union_singleton}
Let $G$ be a 5-regular, 3-sparse graph on at least 16 vertices which contains no 6-set. Let $Y\subset V$, with $|Y|\geq 6$, be a proper core that is also a $j$-set for $j\in \{7,8\}$ and let $\{u,w\}$ be a vertex-pair with $u\in Y$ and $w\notin Y$. Further let $Z\subset V$, with $|Z|\geq 6$, be a blocking $k$-set for $\{u,w\}$ with $k\in \{7,8\}$, such that no vertex in $V-Z$ has more than two neighbours in $Z$ and such that $Y\cap Z=\{u\}$. Then:
\begin{enumerate}
\item[(a)] $d_{G[Z]}(u)=2$, $Z$ is an 8-set and $\emptyset \neq N_{u,w}\subset Y$;
\item[(b)] $Y\cup Z$ is a $(j+5-d(Y,Z))$-set with $d(Y,Z)\leq j-2$;
\end{enumerate}
Furthermore suppose that $Y\cup Z\subsetneq \overline{Y\cup Z} = V$. Then:
\begin{enumerate}
\item[(c)] if $j=7$ and $d(Y,Z)=1$ then $G$ contains an admissible vertex-pair; and
\item[(d)] if $j=8$ and $d(Y,Z)\in \{1,2\}$ then $G$ contains an admissible vertex-pair.
\end{enumerate}
\end{lemma}

\begin{proof}
(a) Since $G$ is 5-regular and $Y$ is a proper core we have $d_{G[Z]}(u)=2$. It follows that $Z$ is an 8-set since $G$ contains no 6-set. Since $Z$ is an 8-set Lemma \ref{lem:nonadmpair} implies that there is a common neighbour and no common neighbour of $u$ and $w$ is in $Z$. Hence, as all neighbours of $u$ are in $Y\cup Z$, we have $\emptyset \neq N_{u,w}\subset Y$.

(b) Now, using (a), 
\begin{eqnarray*}
        i(Y\cup Z) &=& i(Y)+i(Z)+d(Y,Z)-i(Y\cap Z)\\
        &=& 3|Y|-j+3|Z|-8+d(Y,Z)\\
        &=& 3|Y\cup Z|+3|Y\cap Z|-j-8+d(Y,Z)\\
        &=& 3|Y\cup Z|+3-j-8+d(Y,Z)\\
        &=& 3|Y\cup Z|-j-5+d(Y,Z).
    \end{eqnarray*}
Hence $Y\cup Z$ is a $(j+5-d(Y,Z))$-set with $d(Y,Z)\leq j-2$ (since $G$ contains no 6-set).

(c) For the final two parts of the lemma we let $\{r,s,t\}=N(u)\cap Y$ and $\{x,y\}=N(u)\cap Z$. Suppose that $j=7$ and $d(Y,Z)=1$. Since $N_{u,w}\subset Y$ by (a), without loss of generality, we may assume $rw\in E$. Since $d(Y,Z)=1$, $\{t,y\}$ is a vertex-pair and since $Y$ is a proper core and $y$ has three neighbours in $Z-u$, $\{t,y\}$ is reducible. Lemma \ref{lem:nonadmpair} implies that there exists a blocker $X$ with $t,y\in X$ and $X$ is a 7- or 8-set. Suppose $X$ is a 7-set and $u\notin X$. Observe that $t$ has at most one neighbour outside of $Y$. By Lemmas \ref{lem:8set} and \ref{lem:7sets} we now have $|X\cap (Y-u)|\geq 3$ and hence $|X\cap Y|\geq 3$. 
Now Lemma \ref{lem:unionof77} implies that $X\cup Y$ is an $m$-set and $X\cap Y$ is an $n$-set with $m+n\leq 13$, since $uy\in E$. Hence we contradict our hypothesis that $G$ contains no 6-set. 

Hence suppose $u\in X$. By Lemma \ref{lem:7sets} $G[X]$ has minimum degree 3 and hence Lemma \ref{lem:8set} implies that $|X\cap Y|\geq 3$. Lemma \ref{lem:unionof77} implies that $X\cup Y$ is an $m$-set and $X\cap Y$ is an $n$-set with $m+n\leq 14$. It follows that $m=7$ and $n=7$. We now consider the 7-sets $X\cup Y$ and $Z-u$. Since $y$ has at most one neighbour outside of $Z$, we can again apply Lemmas \ref{lem:8set} and \ref{lem:7sets} to deduce that $|(X\cup Y)\cap (Z-u)|\geq 3$. Now Lemma \ref{lem:unionof77} implies that $(X\cup Y)\cup (Z-u)$ is a $m$-set and $(X\cup Y)\cap (Z-u)$ is a $n$-set with $m+n\leq 14$. Our hypothesis that $G$ has no 6-set implies $X\cup Y\cup Z$ is a 7-set. Since $|V|\geq 16$, Lemma \ref{lem:small} implies that $X\cup Y\cup Z\subsetneq V$. However $G$ has no 6-set so it follows that $\overline{X\cup Y\cup Z}\subsetneq V$ and this contradicts the hypothesis that $\overline{Y\cup Z} = V$.

Similarly if $X$ is an 8-set, $t$ has at least four neighbours in $Y$ so $|X\cap Y|\geq 2$. Again if $|X\cap Y|=2$ then we contradict our hypothesis. So $|X\cap Y|\geq 3$. Since $uy\in E$, Lemma \ref{lem:unionof78} now implies that $X\cup Y$ is a $m$-set and $X\cap Y$ is a $n$-set where $m+n\leq 14$.
Hence $X\cup Y$ is 7-set. Then we can consider the 7-sets $X\cup Y$ and $Z-u$ to obtain a contradiction to our hypothesis.

(d) Now suppose that $j=8$ and $d(Y,Z)=1$. (The case when $d(Y,Z)=2$ can be proved similarly, noting that the second edge either simplifies the proof or it is incident to $t$ or $y$ but in that case we can use the vertex-pair $\{s,x\}$ instead.)

We may suppose that $\{t,y\}$ is reducible and non-admissible. Lemma \ref{lem:nonadmpair} implies that there exists a blocker $X$ with $|X|\geq 6$ and $t,y\in X$, such that $X$ is an $\ell$-set with $\ell \in \{7,8\}$. Suppose $\ell=7$ and $u\notin X$. Since $t$ has at least two neighbours in $Y-u$ and at least three neighbours in $X$, we have $|X\cap (Y-u)|\geq 2$. Lemmas \ref{lem:8set} and \ref{lem:7sets} imply that $|X\cap (Y-u)|\geq 3$ and hence $|X\cap Y|\geq 3$. Now, since $uy\in E$, Lemma \ref{lem:unionof78} implies that $X\cup Y$ is a $m$-set and $X\cap Y$ is a $n$-set where $m+n\leq 14$. It follows that $m=7$ and $n=7$. Now consider the 7-sets $X\cup Y$ and $Z-u$. We can now use the same argument deployed in part (c) to contradict the hypothesis that $\overline{Y\cup Z} = V$.

So suppose $u\in X$.  By Lemmas \ref{lem:8set}, \ref{lem:7sets} and our hypothesis $|X\cap Y|\geq 3$. 
Lemma \ref{lem:unionof78} now implies that $X\cup Y$ is a $m$-set, $X\cap Y$ is a $n$-set and $m+n\leq 15$. By our hypothesis $X\cup Y$ is not a 9-set and if $X\cup Y$ is a 7-set, then we obtain the same contradiction as previously. Hence $m=8$ and $n=7$. Now consider $X\cup Y$ and $Z-u$. Again $|(X\cup Y)\cap (Z-u)|\geq 3$. By Lemma \ref{lem:unionof78}, we have that $(X\cup Y)\cup (Z-u)$ is an $m$-set and $(X\cup Y)\cap (Z-u)$ is an $n$-set with $m+n\leq 15$. Hence either $m=7$ or $m=8$. That is $X\cup Y \cup Z$ is a 7- or 8-set. If $X\cup Y \cup Z$ is a 7-set, then we again contradict the hypothesis that $\overline{Y\cup Z} = V$. Similarly if $X\cup Y \cup Z$ is an 8-set and $X\cup Y\cup Z\neq \overline{X\cup Y\cup Z}$ then we contradict the hypothesis that $\overline{Y\cup Z} = V$ using the hypothesis that $G$ contains no 6-set. Hence we have $X\cup Y\cup Z= \overline{X\cup Y\cup Z}=V=\overline{Y\cup Z}\neq Y\cup Z$ and $|V|=16$. However Lemma \ref{lem:trifree} implies that $|Z-u|\geq 9$ and $|Y-u|\geq 8$, which gives a contradiction. 

Hence we may suppose that $\ell=8$ and $u\notin X$. Then, similarly, $|X\cap (Z-u)|\geq 3$ and hence $|X\cap Z|\geq 3$. Since $ut\in E$ with $u\notin X$ and $t\notin Z$, we have $d(X,Z)\geq 1$. Lemma \ref{lem:unionof88} now implies that $X\cup Z$ is an $m$-set and $X\cap Z$ is an n-set with $m+n\leq 15$. Hence either $m=7$ or $m=8$. If $X\cup Z$ is a 7-set, then $|(X\cup Z)\cap Y|\geq 3 $ and Lemma \ref{lem:unionof78} implies that $(X\cup Z)\cup Y$ is a $m$-set and $(X\cup Z)\cap Y$ is a $n$-set where $m+n\leq 15$. If $X\cup Y \cup Z$ is a 7-set, then we again contradict the hypothesis that $\overline{Y\cup Z} = V$. Similarly if $X\cup Y \cup Z$ is an 8-set and $X\cup Y\cup Z\neq \overline{X\cup Y\cup Z}$ then we contradict the hypothesis that $\overline{Y\cup Z} = V$ using the hypothesis that $G$ contains no 6-set. 

Hence we have $X\cup Y\cup Z= \overline{X\cup Y\cup Z}=V=\overline{Y\cup Z}\neq Y\cup Z$ and $|V|=16$. Then we can apply the same argument as in the previous paragraph. 
Hence suppose $X\cup Z$ is an 8-set. Then $|(X\cup Z)\cap Y|\geq 2$. If $|(X\cup Z)\cap Y|=2$, then $(X\cup Z)\cap Y=\{u,t\}$ and $d(X\cup Z, Y)\leq 1$. That is $X\cup Y\cup Z$ is a 10-set or an 11-set. 
If $X\cup Y \cup Z \neq \overline{X\cup Y\cup Z}$ and $X\cup Y \cup Z$ is a 10-set, then, since $|V|\geq 16$ and we must add a degree 5 vertex to obtain $V$, $V$ is an 8-set with $|V|=16$. However Lemma \ref{lem:trifree} implies that $|V|\geq 20$, giving a contradiction. 
If $X\cup Y \cup Z \neq \overline{X\cup Y\cup Z}$ and $X\cup Y \cup Z$ is a 11-set, then, since we must add at least one degree 5 vertex to obtain the closure $|V|\leq 18$. However, by Lemma \ref{lem:trifree}, $V$ has at least 20 vertices, giving a contradiction.

If $X\cup Y\cup Z=\overline{X\cup Y\cup Z}=V$ and $X\cup Y\cup Z$ is a 10-set, then $V$ is a 10-set and $|V|=20$.  Since $(X\cup Z)\cap Y=\{u,t\}$, $X-t$ is a 7-set. We also have that $X\cap Z$ and $Z-u$ are 7-sets. By Lemma \ref{lem:small}, $|(X\cup Z)\setminus\{u,t\}|=9$. However by Lemma \ref{lem:trifree}, we obtain a contradiction. 

Hence suppose $X\cup Y\cup Z=\overline{X\cup Y\cup Z}=V$ and $X\cup Y\cup Z$ is an 11-set, then $V$ is a 11-set and $|V|=22$. Then $|Y\cup Z|\geq 20$, so $|X\setminus (Y\cup Z)|\leq 2$, which gives a contradiction to our assumption that $G$ has no 6-set. 
Hence $|(X\cup Z)\cap Y|\geq 3$. Then Lemma \ref{lem:unionof88} implies that $X\cup Y\cup Z$ is a 7-, 8- or 9-set. If $X\cup Y\cup Z$ is a 7-set we contradict our hypothesis that $G$ has at least 16 vertices. If $X\cup Y \cup Z$ is an 8-set then we can apply the same argument as in the previous paragraph. 
Hence suppose $X\cup Y\cup Z$ is a 9-set. If $X\cup Y \cup Z \neq \overline{X\cup Y\cup Z}$, then we contradict that $|V|\geq 16$, since we must have a degree 5 vertex outside of $X\cup Y\cup Z$ and hence $V$ is at most a 7-set. If $X\cup Y\cup Z=\overline{X\cup Y\cup Z}=V$, then $|V|=18$. However again, by Lemma \ref{lem:trifree}, $|V|\geq 20$, giving a contradiction. 
\end{proof}

\begin{proposition}\label{prop:smallenough2}
Let $G$ be a connected, 5-regular, 3-sparse graph which contains no 6-set. Suppose that every reducible vertex-pair is non-admissible. Let $Y$ be both a proper core and a $j$-set for $j\in \{7,8\}$ and let $\{u,w\}$ be a reducible vertex-pair with $u\in Y$ and $w\notin Y$. Suppose that $Z$ is a blocker for $\{u,w\}$, $|Y\cap Z|=1$ or $Y\cap Z=\{x,y\}$ and $xy\notin E$, and $G[Z]$ does not have minimum degree 3. If $Y\cup Z= \overline{Y\cup Z} = V$ then either $G$ is $\mathcal{R}_3$-independent or $|V|\leq 14$. 
\end{proposition}

\begin{proof}
Suppose $Y\cup Z= \overline{Y\cup Z} = V$ and $G$ is not $\mathcal{R}_3$-independent. We will show that $|V|\leq 14$. First note that $G$ is an $\mathcal{R}_3$-circuit by Corollary \ref{cor:circuit}. Let $Y\cup Z$ be a $k$-set. Since $G$ contains no 6-set we have $k\geq 7$. The conclusion of the proposition follows from Lemma \ref{lem:small} when $k=7$. We complete the proof by showing this equality must hold. Suppose $k>7$.
By our hypotheses, we may suppose $G[Z]$ has a unique vertex of degree less than 3, this vertex has degree 2 and $Z$ is an 8-set.
By the hypotheses $|Y\cap Z|=1$ or $Y\cap Z=\{x,y\}$ and $xy\notin E$.
If $Y\cap Z=\{x,y\}$ and $xy\notin E$, then we contradict 5-regularity since $Z$ has a unique vertex of degree 2. Hence we may assume $Y\cap Z=\{u\}$. 
Then Lemma \ref{lem:union_singleton} (b) implies that $k=j+5-d(Y,Z)$ with $d(Y,Z)\leq j-2$.
Since $|Y\cap Z|=1$, Lemma \ref{lem:connected} implies that $d(Y,Z)\geq 2$. So in each case we have a $k$-set with $k\leq 11$. Let $d:=d(Y,Z)$. We split the remainder of the proof into cases based on $j$.\\

\noindent \textbf{Case 1. $j=7$.}\\
Lemma \ref{lem:union_singleton} (b) implies that $2\leq d\leq 4$. Suppose first that $d=2$. Then let $a \in Y$ be a common neighbour of $\{u,w\}$, since $Y\cup Z = V$. Then $ua \in E$ and observe that $Y\cap (Z\cup \{a\})=\{u,a\}$ with $d(Y, Z\cup \{a\})= 1$. We may now apply Lemma \ref{lem:intbridge} (2) to deduce that $G$ is $\mathcal{R}_3$-independent, contradicting that $G$ is an $\mathcal{R}_3$-circuit. 

Next suppose $d=3$. By Lemma \ref{lem:union_singleton} (b), $Y\cup Z$ is a 9-set, which implies $|V|=18$ by Lemma \ref{lem:small}. By Lemma \ref{lem:trifree}, $|Z-u|\geq 9$ and $|Y-u|\geq 9$, hence $|Y\cup Z|\geq 20$, a contradiction.
Similarly when $d=4$, $|V|=16$ and again Lemma \ref{lem:trifree} implies that $|Y\cup Z|\geq 20$, a contradiction.
\\

\noindent \textbf{Case 2. $j=8$.}\\ 
Lemma \ref{lem:union_singleton} (b) implies that $2\leq d \leq 5$. First suppose $d=2$. Then $\{u,w\}$ has a common neighbour $a\in V-Z$. Since $Y\cup Z=V$, $a\in Y$. Then $au\in E$ and observe that $Y\cap (Z\cup \{a\})=\{u,a\}$ with $d(Y, Z\cup \{a\})=1$. We may now apply Lemma \ref{lem:intbridge} (2) to deduce that $G$ is $\mathcal{R}_3$-independent, contradicting that $G$ is an $\mathcal{R}_3$-circuit.

Now suppose $d=3$. Then again $\{u,w\}$ has a common neighbour $a\in Y$. Observe that $Y\cap (Z\cup \{a\})=\{u,a\}$ with $d(Y, Z\cup \{a\})\leq 2$ and $a$ has two neighbours in $Y\setminus \{u,a\}$, since $Y$ has minimum degree 3. If $a$ has another neighbour in $Z-u$, then $d(Y, Z\cup \{a\})=1$ and we can apply Lemma \ref{lem:intbridge} (2) to obtain a contradiction. Hence suppose $a$ has a third neighbour $t$ in $Y\setminus \{u,a\}$, so $d(Y, Z\cup \{a\})=2$, $Y\setminus \{u,a\}$ is an 8-set and $Z-u$ is a 7-set. Consider the pair $\{t,w\}$. By our hypothesis $tw\notin E$. Thus $\{t,w\}$ is a reducible vertex-pair and we obtain a blocker $T$ by Lemma \ref{lem:nonadmpair}, since every reducible vertex-pair is non-admissible. Suppose first that $T$ is a 7-set and $a\notin T$, then, by Lemma \ref{lem:ess-ec}, $u\in T$. We consider $d_{G[T]}(u)$ and in each case contradict Lemma \ref{lem:ess-ec}. We can apply the same argument if $a\in T$ and $u\notin T$. Hence suppose $a,u\in T$. By our hypothesis $d_{G[T]}(u)\geq 3$ and $d_{G[T]}(a)\geq 3$. If both $u,a$ have degree 3 in $T$, then we contradict our hypothesis. We contradict Lemma \ref{lem:ess-ec} if one of exactly one of $u,a$ has degree 3. If $d_{G[T]}(u)=d_{G[T]}(a)=4$, then again we contradict Lemma \ref{lem:ess-ec}. If exactly one of $u,a$ has degree 5 in $T$, then $T\setminus \{u,a\}$ is a 9-set and again we contradict Lemma \ref{lem:ess-ec}. Hence both $u,a$ have degree 5 in $T$ and $T\setminus \{u,a\}$ is a 10-set and we obtain a 6-set contradicting our hypothesis. Hence $T$ is an 8-set with $|T|\geq 6$, $t,w \in T$ and $a \notin T$. By Lemma \ref{lem:ess-ec}, $u\in T$. Since $a\notin T$, $d_{G[T]}(u)\leq 4$ and $T-u$ always contradicts Lemma \ref{lem:ess-ec}.

Next suppose $d=4$. Again $\{u,w\}$ has at least one common neighbour $a\in Y$. Observe that $Y\cap (Z\cup \{a\})=\{u,a\}$ with $d(Y, Z\cup \{a\})\leq 3$. Suppose $a$ has a second neighbour in $Z$, then $d(Y, Z\cup \{a\})=2$. Then Lemma \ref{lem:trifree} implies $|Z-u|\geq 9$ and $|Y-\{u,a\}|\geq 9$ and hence $|V|\geq 20$, which contradicts that $V$ is a 9-set, by Lemma \ref{lem:union_singleton} (b). 
Hence $a$ has exactly one neighbour in $Z$ and $d(Y, Z\cup \{a\})=3$. Again, by Lemma \ref{lem:trifree}, $|Z-u|\geq 9$ and, since $|Y-u|\geq 8$, we have $|Y-\{u,a\}|\geq 8$. Hence we contradict that $V$ is a 9-set. 

Finally suppose $d=5$. By Lemma \ref{lem:union_singleton} (b), $Y\cup Z$ is an 8-set, which implies $|V|=16$ by Lemma \ref{lem:small}. By Lemma \ref{lem:trifree}, $|Z-u|\geq 9$ and $|Y-u|\geq 8$, hence $|Y\cup Z|\geq 18$, a contradiction.
\end{proof}

\subsection{Proof of the main result}
\label{subsec:proof}

We need one computational result.

\begin{lemma}\label{lem:small_compute}
Let $G$ be connected, 5-regular and 3-sparse on 12 or 14 vertices. Then $G$ is $\mathcal{R}_3$-independent.
\end{lemma}

\begin{proof}
It can be checked, for example using Nauty \cite{Nauty}, that there are 7848 connected, 5-regular graphs on 12 vertices and 3459383 connected, 5-regular graphs on 14 vertices\footnote{Some of these graphs are not 3-sparse, and some do not satisfy the conditions we can prove a minimal counterexample must have, but we did not calculate exactly how many in either case.}. One can compute the rank of the rigidity matrix with symbolic entries in, for example, PyRigi \cite{Pyrigi} to verify that each of these connected, 3-sparse and 5-regular graphs are $\mathcal{R}_3$-independent.    
\end{proof}

We now prove Theorem \ref{thm:main}.

\begin{proof}[Proof of Theorem \ref{thm:main}]
Lemma \ref{lem:max} gives the necessity.
For the sufficiency we suppose the theorem is false and let $G$ be a minimal counterexample. Corollary \ref{cor:circuit} implies that $G$ is a $\mathcal{R}_3$-circuit.
Since $G$ is 5-regular and $3$-sparse, we have $|V|\geq 12$. Lemma \ref{lem:small_compute} now implies that $|V|\geq 16$.

\begin{claim}\label{claim:nonadm}
Every reducible edge and vertex-pair in $G$ is non-admissible.
\end{claim}

\begin{proof}[Proof of Claim \ref{claim:nonadm}.]
Suppose $G/e$ (resp. $G/\{u,v\}$) is 3-sparse. Then by the definition of reducible it is easy to see that $G/e$ (resp. $G/\{u,v\}$) satisfies the hypotheses of 
Theorem \ref{thm:1highdeg} and hence $G/e$  (resp. $G/\{u,v\}$) is $\mathcal{R}_3$-independent. Therefore Lemma \ref{lem:vsplit} implies that $G$ is $\mathcal{R}_3$-independent, a contradiction.
\end{proof}

\begin{claim}\label{claim:6set}
Every 6-set $X\subsetneq V$ in $G$ induces a copy of $K_3$. 
\end{claim}

\begin{proof}[Proof of Claim \ref{claim:6set}.]
We first show that every such 6-set in $G$ induces a complete graph.
Suppose that $X\subsetneq V$ is a 6-set and $G[X]$ is not complete. Then $|X|\geq 5$. Let $t,a,b\in X$ be chosen so that $at\in E$ and $ab\notin E$.
Suppose that $G'=G-at+ab$ is not 3-sparse.
Then there exists $X'\subset V$ with $i_{G'}(X')\geq 3|X'|-5$. Since $G$ is 3-sparse, $a,b\in X'$ and $t\notin X'$, then we contradict the 3-sparsity of $G$ except when $i_{G'}(X')= 3|X'|-5$ and $X'$ is a 6-set in $G$. Since $X'$ is a 6-set containing non-adjacent vertices $a,b$ we have $|X'|\geq 5$. Since $G$ is 3-sparse, $t$ has at most 3 neighbours in $X'$. Hence there exists $c\in X'\setminus \{a,b\}$ such that $ct\notin E$. 

Consider $G''=G-at+tc$. As above there exists a 6-set $X''$ with $c,t\in X''$ and $a\notin X''$. Since $|X''|\geq 5$, $c\in X'\cap X''$ and $d(X',X'')\geq 1$ (as $at\in E$), Lemma \ref{lem:unionof66} gives a contradiction.
Thus at least one of $G-at+ab$ and $G-at+tc$ is a connected 3-sparse graph and hence $\mathcal{R}_3$-independent by Theorem \ref{thm:1highdeg}. By Lemma \ref{lem:isosub}, $G$ is $\mathcal{R}_3$-independent, contradicting the assumption that $G$ is a counterexample to the theorem. Hence $X$ induces a complete graph.

Suppose next, for a contradiction, that one such set $X$ induces a copy of $K_4$. Then Proposition \ref{prop:k4admiss} implies that there is an admissible edge. This contradicts Claim \ref{claim:nonadm}.
\end{proof}

\begin{claim}\label{claim:nok4-e}
$G$ contains no subgraph isomorphic to $K_4-e$.
\end{claim}

\begin{proof}[Proof of Claim \ref{claim:nok4-e}.]
Suppose there exists a subset $X=\{a,b,c,d\}$ of $V$ and $G[X]$ induces $K_4-bc$.
First suppose $ad$ is in exactly two triangles (that is, $ad$ is reducible). Then Claims \ref{claim:nonadm}, \ref{claim:6set} and Lemma \ref{lem:nonadmedge} imply there is a 7-set $A$ with $|A|\geq 6$ such that $a,d \in A$ and $b,c \notin A$. Now $a$ has two neighbours in $A-d$ and $d$ has two neighbours in $A-a$ by 5-regularity. Hence $A\setminus \{a,d\}$ is a 6-set on at least 4 vertices contradicting Claim \ref{claim:6set}.

Thus suppose $ad$ is in at least three triangles (that is, let $f\in N_{a,d}$) and let $g$ be the final neighbour of $a$. Note $bf,cf \notin E$ by Claim \ref{claim:6set}. Since $bf,bc \notin E$, $ab$ is in at most two triangles. As before we can show that $bg\notin E$ and hence $ab$ is in exactly one triangle.
Lemma \ref{lem:nonadmedge} and Claim \ref{claim:6set} imply that there is a 7-set $B$ with $|B|\geq 6$, $a,b\in B$ and $d\notin B$. Lemma \ref{lem:7sets} implies that $a$ and $b$ have degree at least 3 in $G[B]$. Without loss of generality this implies that $c\in B$. Hence $f\notin B$ (otherwise $d$ has four neighbours in $B$ and $B\cup d$ is a 6-set contradicting Claim \ref{claim:6set}). Hence $g\in B$. 

By symmetry with the argument above we can deduce that $fg\notin E$ and there exists a 7-set $C$ with $a,f\in C$ and $d\notin C$. Again, as above, there is no loss in generality in assuming $b\in C$ so $c\notin C$ and hence $g\in C$. Thus $|B\cap C|\geq 3$ and since $gb\notin E$ we have $G[B\cap C]\not\cong K_t$ for any $t$. Then $B\cup C$ has at least 4 vertices and hence $B\cup C$ is not a 6-set. Now Lemma \ref{lem:unionof77} implies that $B\cup C$ is a 7-set. Since $d$ has four neighbours in $B\cup C$, this contradicts Claim \ref{claim:6set}.
\end{proof}

\begin{claim}\label{claim:no6satall}
$G$ contains no proper subset that is a 6-set. 
\end{claim}

\begin{proof}[Proof of Claim \ref{claim:no6satall}]
Suppose for a contradiction that $G$ contains a 6-set $T\subsetneq V$. By Claim \ref{claim:6set} we may suppose $G[T]=K_3$. By Claim \ref{claim:nok4-e} every edge of $T$ is reducible and by Claim \ref{claim:nonadm} every edge of $T$ is non-admissible. Let $T=\{a,b,c\}$. We may now apply Lemma \ref{lem:nonadmedge} to the reducible edges $ab$ and $ac$. By Claim \ref{claim:6set} this gives us a 7-set $X$ with $|X|\geq 6$, $a,b\in X$ and $c\notin X$ and a 7-set $Y$ with $|Y|\geq 6$, $a,c\in Y$ and $b\notin Y$.  
Claim \ref{claim:6set} and the edges $ac,bc$ imply that $c$ has either two or three neighbours in $X$. Hence $X\cup \{c\}$ is either a 7- or 8-set. Similarly $Y\cup \{b\}$ is either a 7- or 8-set. Let $Z=(X\cup \{c\})\cap (Y\cup \{b\})$ and note that $T\subseteq Z$. If $T=Z$ then Lemma \ref{lem:no_triangle_intersect} implies there is a copy of $K_4-e$ in $G$. This contradicts Claim \ref{claim:nok4-e}. So $|Z|\geq 4$ and hence $|X\cap Y|\geq 2$ and since $bc\in E$ we have $d(X,Y)\geq 1$. 
Thus Lemma \ref{lem:unionof77} implies $|X\cap Y|=2$. 
Let $X\cap Y=\{a,t\}$. Then, since $X$ is a 7-set, $X\setminus \{a,t\}$ is a 6-set contradicting Claim \ref{claim:6set}.
\end{proof}

Since $|V|\geq 16$, Lemma \ref{lem:reducible} and Claim \ref{claim:no6satall} imply that there exists a reducible vertex-pair.
Let $\mathcal{X}$ be the collection of proper cores $X \subsetneq V$ such that $X$ is a $j$-set with $j\in \{7,8\}$. 

\begin{claim}\label{claim:nonempty}
$\mathcal{X}\neq \emptyset$. 
\end{claim}

\begin{proof}[Proof of Claim \ref{claim:nonempty}.]
Claim \ref{claim:nonadm} and Lemma \ref{lem:nonadmpair} imply that there exists a $k$-set $X$ such that $|X|\geq 6$ and $k\in \{6,7,8\}$. Suppose $X=V$. Then Lemma \ref{lem:nonadmpair} implies that $k\neq 8$ and then Lemma \ref{lem:small} implies that $|V| \in \{12,14\}$, a contradiction.
Hence it follows that $X\subsetneq V$. Now Claim \ref{claim:no6satall} implies that $k\neq 6$.

If $X$ is a proper core then $X \in \mathcal{X}$, so suppose there exists $v\in X$ such that $d_{G[X]}(v)\leq 2$. 
If $G[X]$ contains a vertex $v$ of degree at most 1 then $k=8$ and $X-v$ is a 6-set contradicting Claim \ref{claim:no6satall}. 
If $G[X]$ contains a vertex $v$ of degree 2 then Lemma \ref{lem:7sets} implies that
$k=8$ and $X-v$ is a 7-set. Suppose $G[X-v]$ contains a vertex $x$ of degree less than 3. The 3-sparsity of $G$ then implies that $x$ has degree 2 and $X\setminus \{v,x\}$ is a 6-set, contradicting Claim \ref{claim:no6satall}.
Thus, if $X$ is not a proper core, we may add the vertices of $V\setminus X$ to $X$ in sequence such that every addition adds at least three edges. 

Suppose $i(V)\geq 3|V|-7$. Then $6|V|-14\leq 2i(V)=5|V|$ which implies that $|V|\leq 14$. This contradicts the assumption that $|V|\geq 16$, and so $i(V)\leq 3|V|-8$. 
As $k\in \{7,8\}$, this further implies that $k=8$. Since $X$ and $V$ are both 8-sets it follows that each addition adds exactly three edges. Hence the final vertex has degree 3, contradicting 5-regularity.
\end{proof}

Let $Y$ be a maximal member of $\mathcal{X}$. 
Lemma \ref{lem:outside2} implies that there exists a reducible vertex-pair $\{u,w\}$ with $u\in Y$ and $w\notin Y$. Claim \ref{claim:nonadm} and Lemma \ref{lem:nonadmpair} imply that there exists a $j$-set $Z$ containing $u,w$ with $j\in \{6,7,8\}$ and $|Z|\geq 6$. 
By Claim \ref{claim:no6satall}, $j\neq 6$. Let $\overline{Y\cup Z}$ be the closure of $Y\cup Z$.
Later we will use the fact that the closure of a subset $T'$ of $T$ is contained in the closure of $T$.

We make explicit here that the claims that follow apply to any $Z$ with the properties in the above paragraph. In particular we may assume no vertex in $V-(Y\cup Z)$ has at least three neighbours in $Z$.

\begin{claim}\label{claim:wneqv}
$\overline{Y\cup Z}\neq V$.
\end{claim}

\begin{proof}[Proof of Claim \ref{claim:wneqv}.]
Suppose for a contradiction that $\overline{Y\cup Z}=V$. Then either $Y\cup Z=\overline{Y\cup Z}$ or $Y\cup Z \subsetneq \overline{Y\cup Z}$. Suppose first that $Y\cup Z=\overline{Y\cup Z}$. 
Thus Propositions \ref{prop:smallenough} and \ref{prop:smallenough2} imply that $|V|\leq 14$, 
contradicting the assumption that $|V|>14$, unless $G[Z]$ has a unique vertex $t$ of degree 2, this vertex is contained in the intersection $Y\cap Z$ and either $Y\cap Z=\{x,y\}$ with $xy\in E$ or $|Y\cap Z|\geq 3.$
However it is easy to see that $Y\cap Z\neq \{x,y\}$ with $xy\in E$, otherwise $Z$ contains a 6-set on at least four vertices. Further if $|Y\cap Z|\geq 3$ and $t$ is contained in $Y\cap Z$ then we may apply a similar argument to Proposition \ref{prop:smallenough}.

Thus we may suppose that $Y\cup Z \subsetneq \overline{Y\cup Z}=V$. Suppose first that $G[Z]$ has minimum degree 3.
Lemmas \ref{lem:unionof77}, \ref{lem:unionof88} and \ref{lem:unionof78} imply that $Y\cup Z$ is a $k$-set with $k\leq 11$. Note that $\overline{Y\cup Z}$ is obtained from $Y\cup Z$ by adding, in order, vertices $v_1, \dots, v_r \in \overline{Y\cup Z}-(Y\cup Z)$ in such a way that the vertex $v_s$ has at least three neighbours in $G[Y\cup Z \cup \bigcup_{i=1}^{s-1} v_i]$. 
Hence $\overline{Y\cup Z}$ is a $j$-set with $j\leq k$. By the maximality of $Y$ and of $Z$, $r\geq 2$.
Now, since $G[\overline{Y\cup Z}]=G$ is 5-regular, $v_r$ has exactly five neighbours in $G[Y\cup Z \cup \bigcup_{i=1}^{r-1} v_i]$ and $v_{r-1}$ has exactly 4 neighbours in $G[Y\cup Z \cup \bigcup_{i=1}^{r-2} v_i]$. This implies that $j\leq k-3\leq 8$. Since $|V|\geq 14$, Lemma \ref{lem:small} implies that $|V|= 16$, $j=8$ and $k=11$. Hence, by reconsidering Lemmas \ref{lem:unionof77}, \ref{lem:unionof88} and \ref{lem:unionof78}, we have that $Y,Z$ are both 8-sets, $Y\cap Z=\{a,b\}$ with $ab\in E$ and $d(Y,Z)=0$. However this is a contradiction because $Z$ is a blocker and hence Lemma \ref{lem:nonadmpair} implies that $u=a$ and $\{u,w\}$ is a vertex-pair with a common neighbour not in $Z$. This neighbour is not in $Y$ since $d(Y,Z)=0$ and hence is in $V-(Y\cup Z)$ contradicting the fact that $a$ has degree 5. 

Now suppose that $G[Z]$ does not have minimum degree 3 and hence has a unique vertex $t$ of degree 2. 
The argument at the beginning of the proof of Proposition \ref{prop:smallenough2} tells us that $Y\cap Z=\{u\}$ and $u$ has degree 2, so $t=u$. It further tells us that $Y\cup Z$ is a $k$-set with $k\leq 13$. 
Claim \ref{claim:no6satall} and Lemma \ref{lem:7sets} imply that $Z$ is an 8-set so Lemma \ref{lem:nonadmpair} implies that $\{u,w\}$ has a common neighbour not in $Z$. Since $u$ has three neighbours in $Y$ this tells us that $d(Y,Z)\geq 1$ and hence in fact $k\leq 12$.
If $d(Y,Z)>2$, then we contradict the assumption that $|V|>14$. In fact $d(Y,Z)=2$ only in the case when $Y$ is an 8-set.
Hence $Y$ is a 7-set and $d(Y,Z)=1$ or $Y$ is a 8-set and $d(Y,Z)\in \{1,2\}$. Then Lemma \ref{lem:union_singleton} (c) and (d) give admissible reductions contradicting Claim \ref{claim:nonadm}.
\end{proof}

\begin{claim}\label{claim:Z8}
$Z$ is an 8-set and the vertex-pair $\{u,w\}$ has  all common neighbours not in $Z$.
\end{claim}

\begin{proof}[Proof of Claim \ref{claim:Z8}.]
Suppose $Z$ is not a proper core. Then $Z$ either contains a 6-set contradicting Claim \ref{claim:no6satall} or $Z$ is an 8-set. So we may suppose $Z$ is a proper core. Further suppose that $Z$ is a 7-set.
Suppose also that $Y$ is a 7-set. By Lemma \ref{lem:unionof77}, if $|Y\cap Z|\geq 3$ then the fact that $i(Y\cap Z)\leq 3|Y\cap Z|-6$ implies that $i(Y\cup Z)\geq 3|Y\cup Z|-8$ and hence $i(\overline{Y\cup Z})\geq 3|\overline{Y\cup Z}|-8$. By Claim \ref{claim:wneqv} this contradicts the maximality of $Y$, so $|Y\cap Z|=2$. 

Hence, by Lemma \ref{lem:unionof77}, we may suppose that $Y\cap Z=\{a,b\}$. 
Then $Z\setminus \{a,b\}$ is a 6-set, contradicting Claim \ref{claim:no6satall}. 
Suppose then that $Y$ is an 8-set. If $|Y\cap Z|=2$, then we obtain the same contradiction as previously. Thus we have $|Y\cap Z|\geq 3$ and $Y\cap Z$ is a 6-set and $d(Y,Z)=0$, contradicting Claim \ref{claim:no6satall}. 

The second conclusion follows immediately from Lemma \ref{lem:nonadmpair}.
\end{proof}

\begin{claim}\label{claim:zfully7}
If $Z$ is a proper core then $Y$ is an 8-set and $|Y\cap Z|=2$.
\end{claim}

\begin{proof}[Proof of Claim \ref{claim:zfully7}.]
Suppose $Z$ is a proper core and $Y$ is a 7-set. If $|Y\cap Z|\geq 3$, then $Y\cap Z$ is a 6-set, contradicting Claim \ref{claim:no6satall}.
If $Y\cap Z=\{a,b\}, ab\in E$, then $Y\setminus \{a,b\}$ is a 6-set, again contradicting Claim \ref{claim:no6satall}. Hence $Y$ is an 8-set.

Suppose next that $|Y\cap Z|\geq 3$.
By Claim \ref{claim:wneqv}, $\overline{Y\cup Z}\neq V$.
If $i(\overline{Y\cup Z})\geq 3|\overline{Y\cup Z}|-8$, then $\overline{Y\cup Z}$ contradicts the maximality of $Y$.
Thus $i(\overline{Y\cup Z})\leq 3|\overline{Y\cup Z}|-9$. Claim \ref{claim:Z8} implies that $Z$ is an 8-set.
Lemma \ref{lem:unionof88} and Claim \ref{claim:no6satall} imply that $Y\cap Z$ is a 7-set. 

Thus suppose $Y\cap Z$ is a 7-set and hence $Y\cup Z$ is a 9-set and $d(Y,Z)=0$.
Let $Y'=Y\setminus (Y\cap Z)$, $Z'=Z\setminus (Y\cap Z)$ and $N=Y\cap Z$.
Suppose, for a contradiction, that $i(Y')\leq 3|Y'|-7$ and $i(Z')\leq 3|Z'|-7$. Hence 
\begin{equation}\label{eq:1}
 d(N, Y')\geq 3|Y|-8 - (3|Y'|-7)- (3|N|-7) = 6.   
\end{equation} 
Similarly $d(N, Z')\geq 6$. However, using Lemma \ref{lem:small} and since $|N|\geq 3$, we have 
\begin{equation}\label{eq:2}
 11\geq 14-|N|=d(N,V\setminus N) \geq d(N, Y') + d(N, Z') \geq 12,   
\end{equation}
a contradiction. Hence we have that at least one of $G[Y']$, $G[Z']$ is a 6-set, contradicting Claim \ref{claim:no6satall}.
\end{proof}

\begin{claim}\label{claim:Znotfullnew}
    $Z$ is not a proper core.
\end{claim}

\begin{proof}[Proof of Claim \ref{claim:Znotfullnew}.]
Suppose $Z$ is a proper core. Claim \ref{claim:zfully7} implies $Y$ is an 8-set and $|Y\cap Z|=2$.
Lemma \ref{lem:unionof88} now gives that $Y\cap Z=\{a,b\}$ with $ab\in E$.
Then $Y\setminus \{a,b\}$ and $Z\setminus \{a,b\}$ are both 7-sets and Lemma \ref{lem:unionof88} implies that $Y\cup Z$ is an $(11-d(Y,Z))$-set with $d(Y,Z)\leq 2$. Note also that without loss of generality we may assume that $a=u$. Recall that $\{u,w\}$ has a common neighbour not in $Z$. By 5-regularity this common neighbour, $c$, is in $Y$.

Hence either $d(Y,Z)=1$ or $d(Y,Z)=2$. Then, without loss of generality, we may choose $d\in Z\setminus Y$ so that $ad\in E$ and $\{c,d\}$ is a vertex-pair with at least one common neighbour. By Claim \ref{claim:Z8}, $wd\notin E$.  
Since $c$ and $d$ have minimum degree 3 in $Y$ and $Z$ respectively, by Lemma \ref{lem:7sets}, $\{c,d\}$ is a reducible vertex-pair with exactly one or two common neighbours. 
Lemma \ref{lem:nonadmpair} implies that there exists a $j$-set $A$ with $|A|\geq 6$, $j\in \{6,7,8\}$ and $c,d\in A$. 

Claim \ref{claim:no6satall} implies that $j\neq 6$ and if $j=7$ then we can replace $Z$ with $A$ and apply Claim \ref{claim:Z8}. Hence $j=8$ and $a\notin A$. If $\{c,d\}$ have a second common neighbour $e$, then $e\notin A$. 

By Claim \ref{claim:no6satall}, $G[Y\setminus \{a,b\}]$ has minimum degree 3. Suppose $|(Y\setminus \{a,b\})\cap A|<3$, then, since $a \notin (Y\setminus \{a,b\})\cup A$, $c$ has at most two neighbours in $A$. Claim \ref{claim:no6satall} further implies, by considering $A-c$, that $c$ has exactly two neighbours $w,\pi$ in $A$. Now $\pi$ has exactly two neighbours in $A-c$ and hence we contradict Claim \ref{claim:no6satall}.
Hence $|(Y\setminus \{a,b\})\cap A|\geq 3$. If $A$ is a proper core then we may use $A$ instead of $Z$ and apply Claim \ref{claim:zfully7} to obtain a contradiction.

So $A$ is not a proper core (recall that Claim \ref{claim:wneqv} shows that the closure of $A\cup Y$ is not $V$). Then there exists a unique vertex $\alpha$ in $A$ with degree 2 in $G[A]$ and $A-\alpha$ is a 7-set. If $\alpha \in A \setminus (Y\cup Z)$ or $\alpha \in Z \setminus Y$ is distinct from $d$ then we may relabel $A$ by $A-\alpha$. If $\alpha =d$ and $A-d\not\subset Y$ then $G[Y\cup (A-d)]$ has minimum degree 3. Lemma \ref{lem:unionof78} and Claim \ref{claim:no6satall} imply that $Y\cup (A-d)$ is a $j$-set for some $j\in \{7,8\}$. 
This contradicts the maximality of $Y$ (again, since Claim \ref{claim:wneqv} shows that the closure of $A\cup Y$ is not $V$). 
If $\alpha =d$ and $A-d\subset Y$ then $d$ has two neighbours in $Y$ contradicting the assumption that $d(Y,Z)\leq 2$. 
If $\alpha \in Y-c$ then $A-\alpha$ is a 7-set blocking the reducible pair $\{c,d\}$. This contradicts Claim \ref{claim:Z8}. 

Hence $\alpha=c$. Now the proper cores $Y\setminus \{a,b\}$ and $A-c$ are both 7-sets intersecting in at least two vertices. If $(Y\setminus \{a,b\}) \cap (A-c)=\{p,q\}$ then, by 5-regularity, the set $Y\setminus \{a,b,p,q\}$ is a 6-set, contradicting Claim \ref{claim:no6satall}. 
Thus $|(Y\setminus \{a,b\}) \cap (A-c)|\geq 3$. Now Lemma \ref{lem:unionof77} implies that imply $(Y\setminus \{a,b\}) \cap (A-c)$ is a 7-set. 
Hence $(Y\setminus \{a,b\}) \cup (A-c)$ is a 7-set by Lemma \ref{lem:unionof77}. Hence $Y \cup (A-c)$ is an 8-set with $|Y \cup (A-c)|>|Y|$. This contradicts the maximality of $Y$ (since Claim \ref{claim:wneqv} shows that the closure of $A\cup Y$ is not $V$) and completes the proof of the claim. 
\end{proof}

By Claim \ref{claim:Znotfullnew}, $Z$ is not a proper core.
Claims \ref{claim:no6satall}, \ref{claim:wneqv} and \ref{claim:Z8} and Lemma \ref{lem:7sets} imply that $Z$ is an 8-set and there is a unique vertex $x\in Z$ that has $d_{G[Z]}(x)=2$. Then $x\in Z \setminus Y$ and $Z-x$ is a 7-set with minimum degree 3. 
Recall that $\{u,w\}$, with $u\in Y\cap Z$ and $w\in Z \setminus Y$, is a reducible vertex-pair with all common neighbours not in $Z$. Suppose $x\neq w$. Remove $x$ from $Z$ to obtain a 7-set $Z-x$ and then $Y\cup (Z-x)$ has minimum degree 3. Since $w\in Z \setminus Y$, $|Y\cup (Z-x)|>|Y|$. By Claim \ref{claim:wneqv}, $\overline{Y\cup Z}\neq V$ and hence this contradicts the maximality of $Y$. 
So we may suppose $x=w$.

\begin{claim}\label{claim:last}
$Z-w\subset Y$.
\end{claim}

\begin{proof}[Proof of Claim \ref{claim:last}.]
Suppose not. Then $|Y\cup (Z-w)|>|Y|$.
By Claim \ref{claim:wneqv} $\overline{Y\cup Z}\neq V$. We may extend $Y\cup (Z-w)$ by adding vertices with at least three neighbours in $Y\cup (Z-w)$ recursively to form $W'$. Since $Z-w\subset Z$, we have $W'\subset \overline{Y\cup Z}$. Hence $W'$ contradicts the maximality of $Y$.
\end{proof}

Note next that all common neighbours of $\{u,w\}$ are not in $Y$. If such a vertex was in $Y$ then $w$ would have three neighbours in $Y\cup Z$ 
and hence $Y\cup Z$ has minimum degree 3 and could be extended, if necessary, to a proper core $W^*$ by Claim \ref{claim:wneqv}. By Claim \ref{claim:last} and the fact that $Y$ is a 7-set or 8-set, $i(W^*)\geq 3|W^*|-8$. Hence we contradict the maximality of $Y$.
Since $G$ is 5-regular and $Y$ is a proper core, $\{u,w\}$ have at most two common neighbours. 

Let $N(w)=\{p,q,r,s,t\}$ where $p,q,r\in V \setminus Y$ and $s,t\in Z-w$. Let $p$ be a common neighbour of $\{u,w\}$ (and if there are two common neighbours then say $q$ is the second). Note also that $su,tu\notin E$ since the common neighbours are outside $Y$.

\begin{claim}\label{claim:lastbut1}
$\{p,t\}$ is a non-admissible vertex-pair.
\end{claim}

\begin{proof}[Proof of Claim \ref{claim:lastbut1}.]
We show that $\{p,t\}$ is a reducible vertex-pair. The claim then follows from  Claim \ref{claim:nonadm}.
By Claim \ref{claim:no6satall}, $pt\notin E$ and suppose that $\{p,t\}$ is not reducible. Since $w$ is a common neighbour they have at least four common neighbours. Since $tu\notin E$ this implies that $p$ has three neighbours in $Z-w$ contradicting the fact that $Y$ is a proper core. Hence $\{p,t\}$ is reducible. 

Now Claim \ref{claim:nonadm} implies that $\{p,t\}$ is non-admissible.
\end{proof}

By Claim \ref{claim:lastbut1}, $\{p,t\}$ is a reducible vertex-pair. Lemma \ref{lem:nonadmpair} implies that there exists a $j$-set $U$ containing $p,t$ with $j\in \{6,7,8\}$ and $|U|\geq 6$. Claim \ref{claim:no6satall} implies that $j\neq 6$ and applying the argument in Claim \ref{claim:Z8} to $U$ implies that $j=8$.
Also if $U$ is a proper core then applying the argument in Claim \ref{claim:Znotfullnew} to $U$ gives a contradiction, so $U$ is not a proper core. 
Thus $U$ is an 8-set, $d_{G[U]}(p)=2$, $U-p$ is both a proper core and a 7-set and $U-p\subset Y$ and the pair $\{p,t\}$ has either one or two common neighbours both in $V \setminus Y$. 

Since $pw\in E$, $G[Y\cup \{p,w\}]$ has minimum degree 3 and contains $Y\cup Z$. Let $W'$ denote the closure of $Y\cup \{p,w\}$. Since $a$ has three neighbours in $Y\cup \{w\}$, $Y\cup \{p,w\}\subseteq \overline{Y\cup Z}$ and hence Claim \ref{claim:wneqv} implies that $W'$ is a proper core. Since $|W'|>|Y|$ this contradicts the maximality of $Y$ whenever $i(W')\geq 3|W'|-8$. In particular this contradiction occurs if $Y$ is a 7-set (so that $Y\cup \{p,w\}$ is an 8-set and $W'$ is a $j$-set for some $j\leq 8$). 
Thus $Y$ is an 8-set, $Y\cup \{p,w\}$ is a 9-set and $W'\neq V$ is also a 9-set. (The fact that $W'\neq V$ is ensured by Claim \ref{claim:wneqv}.)
By Claim \ref{claim:Z8}, $\{q,t\}$ and $\{r,t\}$ are also vertex-pairs. Consider the pair $\{q,t\}$. 
If $\{q,t\}$ is reducible, then by Claim \ref{claim:nonadm} and applying a similar argument to the previous paragraph, we obtain an 8-set and $q$ has exactly two neighbours in $Y$. If $\{q,t\}$ is not reducible, then $\{q,t\}$ has two common neighbours in $Y$ and a single common neighbour not in $Y$, since $Y$ is a proper core, and hence $q$ has exactly two neighbours in $Y$. Thus both $q$ has exactly two neighbours in $Y$. The same argument applies to $r$.

Hence $Y\cup \{p,q,r,w\}$ is a 9-set. It follows that none of $pq,pr,qr$ exists (otherwise $W'$ contradicts the maximality of $Y$). Thus $\{p,q\}$, $\{p,r\}$ and $\{q,r\}$ are vertex-pairs with $w\in N_{p,q},N_{p,r},N_{q,r}$. 
Suppose the pairs $\{q,t\}$ and $\{r,t\}$ are both not reducible. Since $Y$ is a proper core, the pair $\{q,r\}$ has a common neighbour $\gamma \notin Y$ distinct from $w$ and, by 5-regularity, at least one common neighbour $\delta \in Y$. Thus $w,\gamma,\delta \in N_{q,r}$. Now $\{w, \gamma\}$ is a vertex-pair with $q,r,t \in N_{w,\gamma}$. Suppose this vertex-pair is not reducible, so $\{w,\gamma\}$ has at least one more common neighbour. Hence $\gamma$ is adjacent to at least one of $p$ or $s$. Since $\gamma$ would have at least four neighbours in $Y \cup \{w,p,q,r\}$, $W'$ contradicts the maximality of $Y$. 

Hence suppose $\{w,\gamma\}$ is a reducible vertex-pair. By Claim \ref{claim:nonadm}, $\{w,\gamma\}$ is non-admissible. By Lemma \ref{lem:nonadmpair}, there exists a $j$-set $P$ with $|P|\geq 6$ and $j\in \{6,7,8\}$, such that $w,\gamma\in P$ and $q,r,t \notin P$. Applying the argument in Claim \ref{claim:Znotfullnew}, $P$ is not a proper core, which implies $P$ is an 8-set, $d_{G[P]}(\gamma)=2$, $P-\gamma$ is a proper core and a 7-set, and $P-\gamma \subset Y$. Then $\gamma$ has five neighbours in $Y\cup \{w,p,q,r\}$, giving a contradiction. 

Thus at most one of the vertex-pairs $\{q,t\}$ and $\{r,t\}$ is not reducible. By a similar argument, no pair $\{i,j\}$ with $i,j\in \{p,q,r\}$ and $i\neq j$ has a common neighbour outside of $Y$ distinct from $w$. Hence $p,q,r$ each have two unique neighbours in $V\setminus (Y\cup \{w,p,q,r\})$. Let $p$, $q$ and $r$ have neighbours $p_1,p_2$; $q_1,q_2$ and $r_1,r_2$ respectively all in $V\setminus (Y\cup \{w,p,q,r\})$. 

Consider the pair $\{p_1,u\}$ with $p_1u\notin E$ otherwise we contradict Claim \ref{claim:Z8}. If $\{p_1,u\}$ is not reducible, then $p_1$ has two neighbours in $Y$ and $Y\cup \{w,p,q,r,p_1\}$ is a 9-set. If $\{p_1,u\}$ is reducible, then, applying a similar argument to previously, $p_1$ has two neighbours in $Y$. The same holds for $p_2$, $q_1$, $q_2$, $r_1$ and $r_2$. No edge between any pair of these six vertices exists, otherwise $W'$ contradicts the maximality of $Y$. 

Suppose the vertex-pair $\{p_1,p_2\}$ has a common neighbour $\lambda \notin Y$ unique from $p$. Now $\{p,\lambda \}$ is a vertex-pair. If $\{p,\lambda \}$ is not reducible, then $\lambda$ has four neighbours in $Y\cup \{w,p,q,r,p_1,p_2\}$ and $W'$ contradicts the maximality of $Y$. If $\{p,\lambda\}$ is reducible, then we obtain that $\lambda$ has two neighbours in $Y$, again contradicting the maximality of $Y$. Similarly $\{q_1,q_2\}$ and $\{r_1,r_2\}$ do not have common neighbours outside of $Y\cup \{ w,p,q,r,p_1,p_2,q_1,q_2,r_1,r_2 \}$.

Next suppose the vertex-pair $\{p_1,q_1\}$ has a common neighbour $\mu \notin Y$. Then $\{\mu,p\}$ and $\{\mu,q\}$ are vertex-pairs. Consider the pair $\{\mu,p\}$ which has $p_1$ as a common neighbour, and note that $\mu p_2 \notin E$. It follows that $\{\mu,p\}$ is reducible. Then $\mu$ has two neighbours in $Y$ and we contradict the maximality of $Y$, since $Y\cup \{w,p,q,r,p_1,q_1,\mu \}$ gives an 8-set. Hence each of $p_1$, $p_2$, $q_1$, $q_2$, $r_1$ and $r_2$ has two unique neighbours in $V \setminus (Y\cup \{w,p,q,r\})$. Then we can repeat the previous argument with each of these vertices to show that each vertex has two neighbours in $Y$ and two distinct neighbours outside of $Y$. Each pair of these vertices also has no edge between them. Since $G$ is connected, $V$ is finite and $W\neq V$, this process must terminate giving either an admissible vertex-pair or a contradiction.
\end{proof}

\section{Concluding remarks}

As already mentioned the complete bipartite graph $K_{d+2,d+2}$ is a counterexample to the naive extension of Theorem \ref{thm:main} to higher dimensions. Nevertheless it would be interesting to establish a result of the form, if $G$ is $d$-sparse, $(d+2)$-regular and not one of a finite number of exceptions then $G$ is $\mathcal{R}_d$-independent. 

In \cite[Theorem 5.1]{JJbounded} Jackson and Jord\'an proved that if $G$ is a graph such that $i(X)\leq \frac{1}{2}(5|X|-7)$ for all $X\subseteq V$ with $|X|\geq 2$, then $G$ is $\mathcal{R}_3$-independent.  Our main result deals with a family of graphs satisfying $i(X)= \frac{5}{2}|X|$. Hence one may wonder if their result can be improved to cover all graphs with $i(X)\leq \frac{5}{2}|X|$. However such graphs need not be $3$-sparse and the double banana graph satisfies both sparsity conditions. Thus it may be challenging to establish which graphs with such a sparsity are $\mathcal{R}_3$-independent. 

In \cite{DKN} an analogue of Theorem \ref{thm:bounded} was obtained in the context of rigidity under $\ell_p$ norms. That is, where Euclidean distance is replaced by $\ell_p$ distance. In this context the analogue of Lemma \ref{lem:max} says that $\ell_p$-independent graphs satisfy $i(X)\leq d|X|-d$ for all $X\subset V$.
We expect there may be a straightforward adaptation of Theorem \ref{thm:1highdeg} but think it would be interesting to extend our techniques to prove an analogue of Theorem \ref{thm:main} in this context.

\section*{Acknowledgements}

A.\,N.\ was partially supported by EPSRC grant EP/X036723/1. We thank Georg Grasegger for computing $\mathcal{R}_3$-independence for 5-regular, 3-sparse graphs on 12 and 14 vertices.

\bibliographystyle{abbrv}
\def\lfhook#1{\setbox0=\hbox{#1}{\ooalign{\hidewidth
  \lower1.5ex\hbox{'}\hidewidth\crcr\unhbox0}}}

\end{document}